\documentclass[12pt]{amsart}
\usepackage{amssymb,amsmath,amsthm,mathrsfs,color,upgreek,mathtools}
\usepackage[margin=1in]{geometry}
\definecolor{links}{rgb}{0.5,0.2,0}
\definecolor{back}{rgb}{0.8,0.8,0.7}
\definecolor{dgreen}{rgb}{0,0.5,0.2}
\definecolor{indigo}{rgb}{0.4,0,0.75}
\usepackage[colorlinks=true,allcolors=links]{hyperref}
\usepackage[nobysame,initials,non-compressed-cites]{amsrefs}

\usepackage[lining]{libertine}
\usepackage{cabin}
\usepackage[libertine]{newtxmath}
\usepackage[T1]{fontenc}
\mathtoolsset{showonlyrefs}
\usepackage{color}

\newcommand{\Cp}{\mathbf P}
\renewcommand{\O}{\mathscr O}
\renewcommand{\H}{\mathcal H}

\newcommand{\ip}[2]{\left\langle {#1}, {#2} \right\rangle}
\newcommand{\abs}[1]{\left\vert {#1}\right\vert}
\newcommand{\avg}[1]{\left\langle {#1}\right\rangle}
\newcommand{\norm}[1]{\left\Vert {#1}\right\Vert}
\newcommand{\A}{\mathrm{A}}
\newcommand{\RH}{\mathrm{RH}}

\renewcommand{\d}{\mathrm d}

\newcommand{\C}{\mathcal C} 
\renewcommand{\l}{\langle}
\renewcommand{\r}{\rangle}
\newcommand{\ep}{\varepsilon}
\renewcommand{\i}{\mathrm{i}}

\DeclareMathOperator{\VMO}{{VMO}}
\DeclareMathOperator{\e}{{e}}
\DeclareMathOperator{\supp}{{supp}}
\DeclareMathOperator{\diam}{{diam}}

\newcommand{\B}{\mathcal S}
\newcommand{\Bes}{\mathcal B}

\newcounter{thms}

\newcounter{other}
\numberwithin{other}{section}
\newtheorem{proposition}[other]{Proposition}
\newtheorem{theorem}[thms]{Theorem}
\newtheorem*{theorem*}{Theorem}
\newtheorem*{proposition*}{Proposition}
\newtheorem{cor}{Corollary}
\newtheorem*{corollary*}{Corollary}
\numberwithin{cor}{thms}
\newtheorem{lemma}[other]{Lemma}
\theoremstyle{definition}
\newtheorem{remark}[other]{Remark}
\newtheorem{definition}[other]{Definition}

\numberwithin{equation}{section}

\title[Sobolev regularity of quasiregular maps]{Quantitative Sobolev regularity of quasiregular maps}

\author[F. Di Plinio]{Francesco Di Plinio}
\thanks{F. Di Plinio was partially supported by the National Science Foundation under the grant NSF-DMS-200510. This material is based upon work supported by the National Science Foundation under Grant No. DMS-1929284 while this author was in residence at the Institute for Computational and Experimental Research in Mathematics in Providence, RI, during the \textit{Harmonic Analysis and Convexity} program of Fall 2022. F.\ Di Plinio is also partially supported by the FRA 2022 Program of University of Napoli Federico II, project ReSinAPAS - Regularity and Singularity in Analysis, PDEs, and Applied Sciences.}

\address[F. Di Plinio]{Dipartimento di Matematica e Applicazioni, Universit\`a di Napoli \\ \newline \indent Via Cintia, Monte S.\ Angelo 80126 Napoli, Italy}
\email{\href{mailto:francesco.diplinio@unina.it}{\textnormal{francesco.diplinio@unina.it}}}

\author[A. W. Green]{A. Walton Green}
\thanks{A. W. Green's research partially supported by NSF grant NSF-DMS-2202813. }
\author[B. D. Wick]{Brett D. Wick}
\thanks{B. D. Wick's research partially supported in part by NSF grant NSF-DMS-1800057, NSF-DMS-200510, NSF-DMS-2054863 as well as ARC DP 220100285.}

\address[A. W. Green, B. D. Wick]{Department of Mathematics, Washington University in Saint Louis\\ \newline \indent 1 Brookings Drive, Saint Louis, Mo 63130, USA}
\email{\href{mailto:bwick@wustl.edu}{\textnormal{bwick@wustl.edu}}, \href{mailto:awgreen@wustl.edu}{\textnormal{awgreen@wustl.edu}}}


\begin{document}

\begin{abstract} We  quantify the Sobolev space norm of the Beltrami resolvent    $(I- \mu \B)^{-1}$,  where $\B$ is the Beurling-Ahlfors transform, in terms of the corresponding Sobolev space norm of the dilatation   $\mu$ in the critical and supercritical ranges. Our estimate entails as a consequence quantitative self-improvement inequalities of Caccioppoli type  for quasiregular distributions with dilatations in $W^{1,p}$, $p \ge 2$.  Our proof strategy is then adapted to yield quantitative estimates for the resolvent   $(I-\mu \B_\Omega)^{-1}$ of the Beltrami equation on a sufficiently regular domain $\Omega$, with $\mu\in W^{1,p}(\Omega)$. Here, $\B_\Omega$ is the compression of $\B$ to a domain $\Omega$. Our proofs do not rely on the compactness or commutator arguments previously employed in related literature.  Instead, they leverage  the weighted Sobolev estimates for compressions of Calder\'on-Zygmund operators to domains, recently obtained by the authors, to extend the Astala-Iwaniec-Saksman technique to higher regularities.
\end{abstract}

\maketitle

\section{Introduction}
For $K \ge 1$, a \emph{$K$-quasiregular} map $f$ in $W^{1,2}_{\mathrm{loc}}(\mathbb C)$ is a solution of the Beltrami equation
	\begin{equation}\label{eq:belt}\tag{B} \overline{\partial} f(z) = \mu(z) \partial f(z), \quad z \in \mathbb C, \end{equation}
where $\mu$, the dilatation, or Beltrami coefficient of $f$, satisfies
	\begin{equation}\label{e:mu-k} \norm{\mu}_{\infty} = k = \frac{K-1}{K+1} < 1.\end{equation}
The assumption \eqref{e:mu-k} forces \eqref{eq:belt} to be an elliptic PDE. Hence, { for $\mu$ compactly supported} there exists a unique \textit{principal solution} $f$ in $W^{1,2}_{\mathrm{loc}}(\mathbb C)$ which is normalized at $\infty$ by
	\[ f(z) = z + O(z^{-1})\]
and is in fact a homeomorphism. An important feature of $K$-quasiregular maps is their self-improving regularity. First, they automatically belong to the smaller Sobolev space $W^{1,p}_{\mathrm{loc}}(\mathbb C)$ for $p$ up to, but not including, the upper endpoint $p_K>2$ of the \textit{critical interval} $I_K$, cf.\ \eqref{eq:crit} below. On the other hand, if $f$ belongs to $W^{1,q}_{\mathrm{loc}}(\mathbb C)$ for some $q<2$, equation \eqref{eq:belt} may be interpreted in the sense of distributions, {{see \eqref{e:BD} below}  
in which case $f$ is termed \emph{weakly $K$-quasiregular}. And in fact, weakly $K$-quasiregular maps in $W^{1,q}_{\mathrm{loc}}(\mathbb C)$ are $K$-quasiregular if $q$ is larger than or equal to the lower endpoint $q_K<2$, again cf.\ \eqref{eq:crit}. Excluding the endpoint case discussed below, both  facts follow from Caccioppoli inequalities, in turn a consequence of the invertibility of the Beltrami resolvent on $L^p(\mathbb C)$ for $p$ in the critical interval 
	\begin{equation}\label{eq:crit} I_K=(q_K,p_K), \quad q_K \coloneqq  \frac{2K}{K+1}, \quad p_K \coloneqq \frac{2K}{K-1}.\end{equation}
which was established by Astala, Iwaniec, and Saksman in \cite{astala01} relying upon the following key results.
\begin{enumerate}
	\item Astala's area distortion theorem, \cite{astala94}, that is, the principal solution belongs to $W^{1,p}_{\mathrm{loc}}$ for $p<p_K$. As a consequence, suitable powers of its Jacobian determinant are Muckenhoupt $\A_p$ weights.
	\item The Coifman-Fefferman theorem from harmonic analysis, \cite{coifman74}, namely that  Calder\'on-Zygmund singular integral operators are bounded on the weighted $L^p(\omega)$ spaces, $1<p<\infty$, if and only if $\omega$ belongs to the Muckenhoupt class $\A_p$. 
\end{enumerate}
The second point arises in connection with the Beurling-Ahlfors transform $\B$, namely the  Calder\'on-Zygmund operator defined for $f \in \C^\infty_0(\mathbb C)$ by
	\begin{equation}\label{e:Beurling} \B f(z) = -\frac{1}{\pi} \lim_{ \ep \to 0} \int_{  \ep<|z-w| <\frac{1}{\ep}} \frac{f(w)}{(z-w)^2} \, \d w, \quad z \in \mathbb C.\end{equation}
As a consequence of $L^p(\mathbb C)$ {  estimates}, $1<p<\infty$, as well as weak-$L^1$ estimates for maximal truncations of Calder\'on-Zygmund kernels, the above limit   exists for almost every $z \in \mathbb C$, when  $f \in L^p(\mathbb C)$ and  $1\leq p<\infty$. Furthermore $\B$ extends to an isometry on $L^2(\mathbb C)$, and   intertwines the $\partial$ and $\overline{\partial}$ derivatives; see \eqref{e:twine} below. Given $\mu$ satisfying \eqref{e:mu-k}, we refer to the operator $(I-\mu \B)^{-1}$ as the Beltrami resolvent.
The general argument of \cite{astala01} is to recast the norm estimate of $(I- \mu \B)^{-1}$ as an \textit{a priori} estimate for an inhomogeneous Beltrami equation, which, by changing variables, amounts to a weighted estimate for the $\overline{\partial}$ equation, with weight given by an appropriate power of the Jacobian of the principal solution. Therefore, utilizing $\B$, the \textit{a priori} estimate is a consequence of (1) and (2).

The operator $I-\mu \B$ fails to be invertible on $L^p(\mathbb C)$ at the endpoints of the interval $(q_K,p_K)$. However,  in essence quantifying point (2) above, Petermichl and Volberg  \cite{petermichl02} obtained the sharp  estimate $\norm{\B}_{{  L^p(\mathbb C,\omega)\to L^p(\mathbb C,\omega)}} \lesssim [\omega]_{\A_p(\mathbb C)}$ for $2 \le p < \infty$, which allows to approximate the endpoint case and obtain that $I-\mu \B$ remains injective at the upper endpoint.  By duality, this implies that weakly $K$-quasiregular maps belonging to $W^{1,q_K}_{\mathrm{loc}}(\mathbb C)$ are in fact $K$-quasiregular.
Interestingly, the Petermichl-Volberg theorem was one of the motivating factors behind the question of  sharp dependence of the $L^p(\omega)$-norm of a generic Calder\'on-Zygmund operator on the $\A_p$ weight characteristic, culminating in Hyt\"onen's celebrated result \cite{hytonen12}.

\subsection{Sobolev regularity of planar Beltrami equations}
This work focuses on the case when $\mu$ has greater regularity, say in addition to $\|\mu\|_\infty <1$, $\abs{D\mu}=|\partial \mu| + |\overline{\partial}\mu|$ belongs to $L^{p}(\mathbb C)$ for some $1<p<\infty$. Clop et. al. \cite{clop-et-al} have proved an analogue for Sobolev spaces which first generalizes the notion of a weakly quasiregular map to a distributional quasiregular map, which means $f$ is only required to be in $L^q_{\mathrm{loc}}$ for some $1<q<\infty$ and satisfies the distributional version of the Beltrami equation \eqref{eq:belt}; see \eqref{e:BD} below. Their strategy is to prove the desired self-improvement estimates first for quasiconformal $f$ which are then extended to distributional quasiregular maps by factorization and the classical Weyl lemma, that holomorphic distributions are holomorphic functions a.e.. We refer to \cites{baison,clop-tolsa,prats18,baison-frac,cruz2013beltrami} for related works on higher order regularity of quasiregular and quasiconformal maps in the complex plane.

Our first result is  a quantitative version of \cite{clop-et-al} in the critical and supercritical range, which we obtain as a consequence of weighted $W^{1,p}$ bounds for the Beurling-Ahlfors operator, in consonance with the strategy of  \cite{astala01} for the zero-th order problem.

\begin{theorem}\label{thm:local}
Assume the dilatation $\mu \in L^\infty(\mathbb C)\cap W^{1,2}(\mathbb C) $ satisfies \eqref{e:mu-k} for some $K\geq 1$.
Then, for each $1<r<2$, 
\begin{equation}\label{e:local-crit} \left\Vert (I- \mu \B)^{-1} \right\Vert_{{  W^{1,r}(\mathbb C)\to W^{1,r}(\mathbb C)}} \lesssim 1\end{equation}
with  implicit constant depending exponentially on $K$, $\|\mu\|_{W^{1,2}(\mathbb C)} $ and $\frac{1}{\min\left\{2-r,r-1\right\}}$.

If in addition $\mu \in W^{1,p}(\mathbb C)$ for some $2<p<\infty$, 
	\begin{equation}\label{e:local-sup} \left\Vert (I- \mu \B)^{-1} \right\Vert_{{  W^{1,p}(\mathbb C)\to W^{1,p}(\mathbb C)}} \lesssim 1+\|\mu\|_{W^{1,p}(\mathbb C)}^2\end{equation}
with  implicit constant depending exponentially on $K$, $\|\mu\|_{W^{1,2}(\mathbb C)} $ and $\max\left\{\frac{1}{p-2},p\right\}$.	
\end{theorem}

In the critical case \eqref{e:local-crit}, one cannot hope for $I-\mu \B$ to be invertible on $W^{1,2}(\mathbb C)$ because in the corollary below, \eqref{eq:cacc-2-sup} fails for $p=2$. Indeed, from \cite{clop-et-al}*{pp. 205-206}, one can consider the quasiregular distribution $\phi(z) = z(1-\log \abs{z})$ which does not belong to $W^{2,2}_{\mathrm{loc}}(\mathbb C)$, though its Beltrami coefficient, $\mu(z) = \frac{z}{\bar z} \frac{1}{2 \log \abs{z}-1}$ does in fact belong to {  $W^{1,2}_{\mathrm{loc}}(\mathbb C)$}. 

Quantitative self-improvement of quasiregular maps is often expressed through Caccioppoli inequalities (see the survey \cite{sbordone} or \cite{astala-book}*{\S 5.4.1}). In our case, Theorem \ref{thm:local} implies the following Caccioppoli inequalities for quasiregular distributions (see \eqref{eq:cacc-1} - \eqref{eq:cacc-2-sup} below). Given an open set $\Omega$ and { $\mu \in W^{1,p}_{\mathrm{loc}} (\Omega) \cap L^\infty(\Omega)$, say $f \in L^{\frac{p}{p-1}}_{\mathrm{loc}}(\Omega)$ is a $\mu$-quasiregular distribution if its Beltrami distributional derivative $(\overline{\partial}-\mu\partial)f=0$. Precisely , following e.g \cite[p. 200]{clop-et-al}, by this we mean that for all $\psi$ in $\C^\infty_0(\Omega)$,
	\begin{equation}\label{e:BD}\left \langle \overline{\partial}f - \mu \partial f,\psi \right \rangle = -\left \langle f, \overline{\partial} \psi - \partial(\mu \psi) \right \rangle = \left \langle f, \overline{\partial} \psi \right\rangle - \left \langle f \partial \mu, \psi \right \rangle - \left \langle f \mu, \partial \psi \right \rangle =0. \end{equation}}

\begin{cor}\label{cor:cacc}
Let $\mu \in L^\infty(\Omega) \cap W^{1,2}_{\mathrm{loc}}(\Omega)$ satisfy  \eqref{e:mu-k} for some $K\geq 1$. Then, for $2<q<\infty$, $f \in L^q_{\mathrm {loc}}(\Omega)$ satisfying \eqref{e:BD}, and any $\eta \in \C^\infty_0(\Omega)$,
	\begin{align}\label{eq:cacc-1} \|\eta (Df)\|_{L^q} &\lesssim \|(D\eta)f\|_{L^q}; \\ 
	\label{eq:cacc-2} \mbox{for all }1<r<2, \quad \|\eta (D^2 f)\|_{L^r} &\lesssim \|(D \eta)f\|_{L^r} + \|(D\eta)(Df)\|_{L^r} + \|(D^2\eta)f\|_{L^r}. \end{align}
In particular, $f \in W^{2,r}_{\mathrm{loc}}(\Omega)$ for every $r<2$.
If furthermore, $\mu \in W^{1,p}_{\mathrm{loc}}(\Omega)$ for some $p>2$, then, for $\frac{p}{p-1}\le q<\infty$, $f \in L^q_{\mathrm {loc}}(\Omega)$ satisfying \eqref{e:BD}, and any $\eta \in \C^\infty_0(\Omega)$,
	\begin{align}\label{eq:cacc-1-sup} \|\eta (Df)\|_{L^q} &\lesssim \|(D\eta)f\|_{L^q}; \\ 
	\label{eq:cacc-2-sup} \mbox{for all }1<r \le p, \quad \|\eta (D^2 f)\|_{L^r} &\lesssim \|(D \eta)f\|_{L^r} + \|(D\eta)(Df)\|_{L^r} +  \|(D^2\eta)f\|_{L^r}. \end{align}
In particular, $f \in W^{2,p}_{\mathrm{loc}}(\Omega)$. 
\end{cor}
Experts in the area will readily observe that without the precise dependence on $\|\mu\|_{W^{1,2}(\mathbb C)}$  we provide, the inequalities \eqref{eq:cacc-1} and \eqref{eq:cacc-1-sup} follow from the fact that $W^{1,2}(\mathbb C)$ embeds into the space of functions with vanishing mean oscillation, or $\VMO(\mathbb C)$, together with the invertibility of $I-\mu \B$ on $L^p(\mathbb C)$ for every $1<p<\infty$ whenever $\mu \in \VMO(\mathbb C)$, see \cite{astala01}*{Theorem 5}. However, by sharpening the assumption to $\mu \in W^{1,2}(\mathbb C)$, we establish, in Lemma \ref{lemma:invLp} below, a quantitative version of this result which is a crucial step in Theorem \ref{thm:local} and Corollary \ref{cor:cacc}.
The conclusion that quasiregular distributions in $L^q_{\mathrm{loc}}(\Omega)$ self-improve to membership in $W^{2,r}_{\mathrm{loc}}(\Omega)$ for $q$ and $r$ in the above specified ranges is one of the main results of Clop et. al. in \cite{clop-et-al}. However, the Caccioppoli inequalities \eqref{eq:cacc-2} and \eqref{eq:cacc-2-sup}, with the precise implicit dependence on the local regularity of $\mu$ inherited from \eqref{e:local-crit} and \eqref{e:local-sup} in
Theorem \ref{thm:local} are new to the best of our knowledge. Theorem \ref{thm:local} and the Caccioppoli inequalities are proved in \S \ref{s:local}, and rely on a Moser-Trudinger estimate for the Jacobian of the principal solution to \eqref{eq:belt} with $\mu \in W^{1,2}(\mathbb C)$. 

\subsection{Global Sobolev regularity of Beltrami equations on domains $\Omega$} Theorem \ref{thm:local} was actually uncovered in our attempts to address a more delicate problem, the invertibility of $I- \mu \B_\Omega$, where $\B_\Omega$ is the compression of the Beurling-Ahlfors transform to a bounded Lipschitz domain $\Omega \subset \mathbb C$ defined by 
	\[ \ip{ \B_\Omega f}{g } = \ip{\B \left( f\mathbf 1_{\overline{\Omega}} \right)}{\mathbf 1_{\overline{\Omega}}g}, \quad f,g \in \C^\infty_0(\mathbb C) .\]
In \cite{diplinio23domains}, we developed new $T(1)$-type theorems and weighted Sobolev space estimates for Cal\-de\-r\'on-Zygmund operators on domains, a broad class which includes compressions of global CZ operators. In particular, together with past work of e.g.\ Tolsa \cite{tolsa2013regularity}, the estimates of \cite{diplinio23domains} uncover the precise connection between boundary regularity of $\Omega$ and weighted  Sobolev estimates for $\B_\Omega$.  In this article, this connection is exploited to extend the resolvent strategy of \cite{astala01} to the Sobolev case and obtain the first quantitative Sobolev estimate for $(I - \mu \B_\Omega)^{-1}$.

The compressed Beltrami resolvent $(I-\mu\B_\Omega)^{-1}$ is connected to the Beltrami equation \eqref{eq:belt} for dilatations $\mu$ whose support is contained in $\overline{\Omega}$ and belonging to  $ W^{1,p}(\Omega)$ for some $p>2$. 
The Caccioppoli inequalities of Corollary \ref{cor:cacc} imply that any solution $f$ to \eqref{eq:belt} with $\mu$ of this form belongs to $W^{2,p}_{\mathrm{loc}}(\Omega)$. Thus, the interest is in global regularity, i.e. whether $f$ belongs to $W^{2,p}(\Omega)$. This problem is even of interest when $f$ is the principal solution, in which case one has the representation from \cite{astala-book}*{p. 165},
	\begin{equation}\label{e:principal-rep-1} \overline{\partial} f = (I - \mu \B_\Omega)^{-1} \mu.\end{equation}
Furthermore, since $\partial f = \B_\Omega( \overline{\partial} f)$ by \eqref{e:twine} below, 
	\begin{equation}\label{e:standard} 	
		\left\Vert {f} \right\Vert_{W^{2,p}(\Omega)} 
		\lesssim  \left(1 + \left\Vert \B_\Omega\right\Vert_{{  W^{1,p}(\Omega)\to W^{1,p}(\Omega)}} \right)
		\left\Vert (I_\Omega-\mu \B_\Omega)^{-1}  \right\Vert_{{  W^{1,p}(\Omega)\to W^{1,p}(\Omega)}} \left\Vert \mu \right\Vert_{W^{1,p}(\Omega)}.
	\end{equation}
The first factor, the norm of $\B_\Omega$ on the Sobolev space $W^{1,p}(\Omega)$ is now well-understood in the supercritical range in terms of the boundary regularity of $\Omega$, see \cites{cruz-tolsa,tolsa2013regularity,prats17}. In fact, by these results, it is quantitatively equivalent to the Besov space $B^{1-\frac{1}{p}}_{p,p}(\partial \Omega)$ norm of the boundary normal of $\Omega$; see Definition \ref{def:Bp} below. Accordingly, we say $\Omega$ is a $\Bes_p$ domain if this boundary regularity condition is satisfied.

The second factor in \eqref{e:standard} is our chief object of interest. While quantitative estimates of this norm appear to be unavailable in past literature, several results of qualitative nature have been obtained through methods in antithesis with those developed herein. Initially, invertibility of $I-\mu \B_\Omega$ was studied in the H\"older scale in \cite{mateu-extra} and subsequently extended to the Sobolev and Triebel-Lizorkin scales, \cites{cruz2013beltrami,prats17,prats19,astala-prats-saksman}. These works all share the Neumann series blueprint initially introduced by Iwaniec in \cite{iwaniec}. The main ingredients are unweighted bounds for $\B_\Omega$ on smoothness spaces, established by means of \emph{unweighted} $T(1)$-type theorems.

The apex of this line of attack was a pair of papers by M.~Prats in \cites{prats17,prats19}, sharpening the result of \cite{cruz2013beltrami}, and establishing among other results the remarkable qualitative fact that $I-\mu \B_\Omega$ is invertible on $W^{1,p}(\Omega)$ assuming only that $\Omega$ is a $\Bes_p$ domain.

\begin{theorem}\label{thm:global}
Let $p\ge r>2$, and $\Omega \subset \mathbb C$ be a bounded simply connected $\Bes_p$ domain. Let $\mu$ be supported in $\overline{\Omega}$, satisfying \eqref{e:mu-k} for some $K\geq 1$, and in addition $\mu \in W^{1,p}(\Omega)$. Let $f$ be the principal solution to \eqref{eq:belt}. Then, for $O=f(\Omega)$ and $\omega=\abs{Jf^{-1}}^{1-p}$,
	\[ \begin{aligned}&\|(I- \mu \B_\Omega)^{-1}\|_{{  W^{1,p}(\Omega)\to W^{1,p}(\Omega)}} \lesssim \O \left[ \left\Vert \B_O \right \Vert_{{  W^{1,p}(O,\omega)\to W^{1,p}(O,\omega)}} + \O^3\left(1 + \|\mu\|_{W^{1,p}(\Omega)}^6\right) \right], \\
	&\O = 1 + \norm{O}_{\Bes_p} + \norm{\Omega}_{\Bes_p}. \end{aligned}
\]
The implicit constant depends 
double exponentially on { $\norm{\mu}_{W^{1,r}(\Omega)}$}, $\max\left\{\frac{1}{r-2},p\right\}$,  $K$, $\norm{\mu}_{W^{1,2}(\Omega)}$, and the Dini character of $O$ and $\Omega$.
\end{theorem}

Using the novel weighted $T(1)$ theorems on domains established by the authors in \cite{diplinio23domains}, and the relationship between these testing conditions and boundary smoothness developed in \cites{cruz-tolsa,tolsa2013regularity}, $\norm{\B_O}_{{  W^{1,p}(O,\omega)\to W^{1,p}(O,\omega)}}$ can be quantitatively controlled by $\norm{O}_{\Bes_{p+\ep}}$ (see Lemma \ref{lemma:Bp}.iii below) for any $\ep>0$, yielding the following corollary.

\begin{cor}\label{c:global} Let $p$, $r$, $\Omega$, $\mu$, $f$, $O$, and $\O$ be as in Theorem \ref{thm:global}. Then, for any $\ep>0$,
		\[ \|(I- \mu \B_\Omega)^{-1}\|_{{  W^{1,p}(\Omega)\to W^{1,p}(\Omega)}} \lesssim \O \left[ \norm{O}_{\Bes_{p+\ep}} + \O^3\left(1 + \|\mu\|_{W^{1,p}(\Omega)}^6\right) \right].
\]
The implicit constant depends on the same parameters of Theorem \ref{thm:global}, as well as  $\ep^{-1}$.
\end{cor} 

Let us provide a more specific description of the relation between Theorem \ref{thm:global}, as well as Corollary \ref{c:global}, and the results of \cite{prats19}. In particular, \cite{prats19}*{Theorem 1.1} tells us that if $\Omega\in \mathcal B_p$ and $\mu\in W^{1,p}(\Omega)$, then the principal solution $f$ lies in $W^{2,p}(\Omega)$.
Standard trace results \cites{gagliardo,triebel83} then entail  that  $O=f(\Omega)$ is a $\Bes_p$ domain as well, which in turn is qualitatively equivalent, see Lemma \ref{lemma:sharp} below, to $\B_O:W^{1,p}(O,\omega) \to W^{1,p}(O,\omega)$, where $\omega$ is as in Theorem \ref{thm:global}. Thus, Theorem \ref{thm:global} holds under the same assumptions  as  \cite{prats19}*{Theorem 1.1, $n=1$},  and may be viewed as a strict quantification  of that result. Furthermore, Corollary \ref{c:global} replaces the analytic condition on $\B_O$ with a fully geometric testing condition on $O$, namely its membership to $\Bes_q$ for some $q>p$, thus providing an explicit dependence on the data $\mu,\Omega $ and $O$.

We close this introduction with  a circle of questions motivated by   Corollary \ref{c:global}. First of all, a crude version of Corollary \ref{c:global} with $\ep=0$ can be obtained without weighted estimates  {  at the price of  exponential dependence on the data $\norm{\Omega}_{\Bes_p}$ and $\norm{O}_{\Bes_p}$; cf.\ Remark \ref{rem:unweight} below. It is thus natural to ask whether} a version of Corollary \ref{c:global} holds with $\ep=0$ and uniform polynomial estimates in the sharp Besov norms on $\Omega$ and $O$. This would hold if the Jacobian power $\omega$ were an $\A_1(O)$ weight polynomially in $\O$ and $\|\mu\|_{W^{1,p}(\Omega)}$, and $\|\B_O \mathbf{1} \|_{W^{1,p}(v,O)}$ were controlled by a constant depending only on $[v]_{\A_1(O)}$ and on the $\Bes_p$ character of $O$. The latter statement for $v=\mathbf 1$ is the content of \cite{cruz-tolsa}*{Theorem 1.1}, whence it is legitimate to ask whether Lebesgue measure can be replaced with a generic $\A_1$ weight therein, and whether a full analogue of  Corollary \ref{c:global} holds for $\ep=0$. 

Furthermore, let us propose a strategy for removing the dependence on the auxiliary $W^{1,r}$-norm of $\mu$ in both Theorem \ref{thm:global} and Corollary \ref{c:global}. { It is introduced to obtain bilipschitz estimates on the principal solution $G$ of an extension of $\mu$ in Proposition \ref{prop:principal} below. A bound on the Lipschitz constant of $G$ (and $G^{-1}$ after a change of variable; cf. \cite[Theorem 5.5.6]{astala-book}) is provided by Theorem \ref{thm:local} due to Sobolev embedding. For this upper bound,} the space $W^{1,r}(\Omega)$ can actually be replaced by any space $X(\Omega)$  enjoying both properties
	\begin{itemize}
	\item[a.] $X(\Omega)$ continuously embeds into both $W^{1,2}(\Omega)$ and $L^\infty(\Omega)$;
	\item[b.] For some $Y \in \{X(\mathbb C),W^{1,2}(\mathbb C)\}$, there exists $M>0$ such that
	\[ \norm{(I-\mu \B)^{-1}}_{Y \to Y} \lesssim 1 + \norm{\mu}_{X(\mathbb C)}^M, \]
	\end{itemize}
By Theorem \ref{thm:local}, $X = W^{1,r}$ for $r>2$ satisfies these conditions. A candidate for a space $X$ larger than $W^{1,r}$ is the Lorentz-Sobolev space consisting of $L^2$ functions with derivatives in the Lorentz space $L^{2,1}$. While it is known that $I-\mu \B$ is invertible on this space \cite{cruz2013beltrami}, no norm estimates are known. This leads to ask whether there is a version of Theorem \ref{thm:local} for $D\mu \in L^{2,1}(\mathbb C)$. 

\subsection*{Structure of the article} In Section \ref{s:quant}, after a few preliminaries, we deduce quantitative estimates for the Beltrami resolvent associated to dilatations $\mu\in W^{1,2}(\mathbb C)$ from a precise $\A_p$-class embedding for powers of Jacobians of the corresponding principal solutions, see Lemma \ref{lemma:invLp} and \ref{lemma:moser} respectively. Section \ref{s:local} contains the proofs of Theorem \ref{thm:local} and Corollary \ref{cor:cacc}. In Section \ref{s:global},  we provide proofs of Theorem \ref{thm:global} and Corollary \ref{c:global}. As an intermediate step, in Proposition \ref{prop:principal}, we establish a quantitative version of a recent result of Astala, Prats and Saksman  of \cite{astala-prats-saksman}*{Theorem 1.1} on the regularity of quasiconformal solutions to Beltrami equations on $\Bes_p$ domains. Furthermore, we explain why methods that do not treat Jacobians of principal solutions as Muckenhoupt weights lead to exponential type estimates in the data $\mu$, $\Omega$, and $O$, so that Theorem \ref{thm:global} is not within their  reach, cf. Remark \ref{rem:unweight} below. Finally, Section \ref{s:Jf} deals with the technical proof of Lemma \ref{lemma:moser}.

\section{The Beltrami resolvent when $\mu\in W^{1,2}(\mathbb C)$}\label{s:quant}
To prove Theorem \ref{thm:local}, we will need a few preliminaries. The facts that we will need from the classical theory of quasiconformal maps will be recalled throughout from the monograph \cite{astala-book}. 
Recall the definition of the Beurling-Ahlfors transform from \eqref{e:Beurling}. It is of particular use because it intertwines the derivatives $\partial:= \frac{\partial}{\partial z}$ and $\overline{\partial}:= \frac{\partial}{\partial {\bar z}}$, which means
	\begin{equation}\label{e:twine} \B(\overline{\partial} f) = \partial f, \quad f \in W^{1,2}(\mathbb C). \end{equation}
This property can be established by appealing to the Fourier transform, or through the Cauchy transform $\mathcal K$ defined by
 	\begin{equation}\label{e:Cauchy} \mathcal K f(z) = \frac{1}{\pi}\lim_{ \ep \to 0} \int_{|z-w| > \ep} \frac{f(w)}{z-w} \, \d w. \end{equation}
The Cauchy transform is the inverse of the $\overline{\partial}$ operator and $\partial \mathcal K = \B$, so that $\partial f = \partial \mathcal K (\overline{\partial} f) = \B (\overline{\partial} f)$.

Use $Df$ to denote the gradient $(\partial f,\overline{\partial}f)$ and for each integer $n \ge 2$, let $D^n f$ denote the vector function consisting of all combinations of $n$-th order partial derivatives in $z$ and $\overline{z}$ of $f$. $D^1f=Df$ and $D^0f=f$. We will use $|D^nf|$ to denote the $\ell^1$ norm of this vector. 

Given an open set $E \subset \mathbb C$, an a.e. positive element of $L^1_{\mathrm{loc}}(E)$ is called a weight on $E$. For $\omega$ a weight on $E$, $n$ a nonnegative integer, and $0<p< \infty$, define the homogeneous and inhomogeneous weighted Sobolev norms by
	\[ \norm{f}_{\dot{W}^{n,p}(E,\omega)} = \sum_{|\alpha|=n} \norm{\left( \partial^{\alpha_1}\overline{\partial}^{\alpha_2} f\right) \omega^{\frac 1p}}_{L^p(E)}, \quad \norm{f}_{{W}^{n,p}(E,\omega)} = \sum_{j=0}^n \norm{f}_{\dot{W}^{j,p}(E,\omega)},\]
where $\alpha = (\alpha_1,\alpha_2) \in \mathbb N^2$ and $|\alpha| = \alpha_1+\alpha_2$.   {  Here, and through out the text we let $\mathcal{L}(X)$ denote the bounded linear operators on the Banach space $X$.  In most applications of this notation $X$ will be either a Lebesgue or Sobolev space, or a weighted version of one of these spaces.}
We also use the local average notation for a cube $Q \subset\mathbb C$,
	\[\avg{f}_{p,Q} = \left( |Q|^{-1} \int_Q \abs{f(z)}^p \, \d z \right)^{\frac 1p},\]
with the simplification $\avg{f}_Q = \avg{f}_{1,Q}$ when $p=1$. We say a weight $\omega$ on $\mathbb C$ belongs to the Muckenhoupt class $\A_p(\mathbb C)$ if the associated characteristic,
	\begin{equation}\label{def:Ap} [\omega]_{\A_p(\mathbb C)} = \sup_{Q \, \mathrm{cube} \, \mathrm{in} \, \mathbb C} \avg{\omega}_Q \avg{\omega^{-1}}_{\frac{1}{p-1},Q}\end{equation}
is finite. To apply the strategy of \cite{astala01}, we will need now weighted Sobolev estimates for $\B$, which were recently obtained for smooth Calder\'on-Zygmund operators in sharp quantitative form in \cite{diplinio22wrt}.  The estimates we require are summarized in the following proposition.
\begin{proposition}\label{p:BC}
Let $n \in \mathbb N$, $1<p<\infty$. There exists $C_{p,n}>0$ such that for any $\omega \in \A_p(\mathbb C)$, 
	\begin{equation}\label{e:Bw} \norm{\B f}_{\dot W^{n,p}(\mathbb C,\omega)} \le C_{p,n} [\omega]^{\max\{1,\frac{1}{p-1}\}}_{\A_p(\mathbb C)} \norm {f}_{\dot W^{n,p}(\mathbb C,\omega)}.\end{equation}
\end{proposition}
\begin{proof}
Since $\B$ is of convolution type, \eqref{e:Bw} more or less follows from the case $n=0$, which is well-known, \cite{petermichl02}, though some care must be taken with the principal value integral. So, one can consult \cite{diplinio22wrt}*{Corollary A.1} for a complete proof of \eqref{e:Bw}.
\end{proof}

Introduce the notation 
	\begin{equation}\label{e:JD}\abs{Jf} = \abs{\partial f}^2 - \abs{\overline{\partial} f }^2,
\end{equation}
which is equal to the determinant of the Jacobian of $f$ as a mapping from $\mathbb R^2$ to itself. A characterization of $K$-quasiconformal mappings equivalent to \eqref{eq:belt} is the distortion inequality
	\begin{equation}\label{e:Jf-quasi} \abs{Df}^2 \le K \abs{Jf}.\end{equation}
The main lemma concerning $\abs{Jf}$ for $\mu \in W^{1,2}(\mathbb C)$ is a consequence of the critical Moser-Trudinger Sobolev embedding, and is proved in \S \ref{ss:moser} and \ref{ss:particular}.

\begin{lemma}\label{lemma:moser}
Suppose $\mu$ is { compactly supported}, satisfies \eqref{e:mu-k} for some $K\geq 1$, and in addition belongs to $W^{1,2}(\mathbb C)$ with $\|\mu\|_{W^{1,2} {  (\mathbb{C})}}\leq L$. Let $f$ be the principal solution to \eqref{eq:belt}, $a \in \mathbb R$, and $1<p<\infty$. Then, the Jacobians $\abs{Jf}^a$ and $\abs{Jf^{-1}}^a$ are both $\A_p(\mathbb C)$ weights. In particular, there exists a constant $C=C(K)>0$ such that for any $1<p<\infty$,
	\begin{align}
	\label{e:Jf-Ap-1} \left[ \abs{Jf^{-1}}^{1-\frac{p}{2}} \right]_{\A_p(\mathbb C)}^{\max\left\{1,\frac{1}{p-1} \right\} }&\le C \exp \left( C \max\left\{p,\tfrac{1}{p-1}\right\}^2 L^2 \right); \\
	\label{e:Jf-Ap-2} \left[ \abs{Jf^{-1}}^{1-p} \right]_{\A_p(\mathbb C)}^{\max\left\{1,\frac{1}{p-1} \right\}} &\le C \exp\left(C \max\left\{p^2,\tfrac{1}{p-1}\right\} L^2 \right).
	\end{align}
\end{lemma}

The second lemma we will use follows from Lemma \ref{lemma:moser} and the strategy of \cite{astala01}. The estimate \eqref{e:invLp} below is known qualitatively since $W^{1,2}(\mathbb C)$ embeds into $\mathrm{VMO}(\mathbb C)$, the functions with vanishing mean oscillation. And it is well-known that $I-\mu \B$ is invertible on all $L^p(\mathbb C)$ for $1<p<\infty$ and $\mu \in \mathrm{VMO}(\mathbb C)$, a result that be found in \cite{astala01}*{Theorem 5}.

\begin{lemma}\label{lemma:invLp}
Suppose $\mu$ satisfies \eqref{e:mu-k} for some $K\geq 1$, and in addition that $\mu \in W^{1,2}(\mathbb C)$ with $\|\mu\|_{W^{1,2} { (\mathbb{C}})}\leq L$. Then, there exists $C=C(K)>0$ such that for all $1<p<\infty$,
	\begin{equation}\label{e:invLp} \norm{ (I- \mu \B)^{-1}}_{\mathcal L(L^p(\mathbb C))} \le C \exp \left( C \max\left\{p,\tfrac{1}{p-1}\right\}^2 L^2 \right).\end{equation}
\end{lemma}
\begin{proof}[Proof of Lemma \ref{lemma:invLp}]
{ An approximation argument such as \eqref{e:triangle} below shows that we may assume $\mu$ is compactly supported.}The Astala-Iwaniec-Saksman strategy from \cite{astala01} shows that
	\begin{equation}\label{e:AIS}  \|(I- \mu \B)^{-1} \|_{\mathcal L(L^p(\mathbb C))} \lesssim_K \| \B\|_{\mathcal L(L^p(\mathbb C,\omega))}, \end{equation}
where $\omega = \abs{J{f^{-1}}}^{1-\frac{p}{2}}$ and $f$ is the principal solution to \eqref{eq:belt}. See \cite{astala-book}*{\S 14.2 (14.25)} for this exact statement, or refer to the proof of Proposition \ref{prop:LpBO} below. Estimating the right hand side of \eqref{e:AIS} by \eqref{e:Bw} in Proposition \ref{p:BC} with $n=0$, and \eqref{e:Jf-Ap-1} in Lemma \ref{lemma:moser} concludes the proof.
\end{proof}

\section{Proof of Theorem \ref{thm:local} and Corollary \ref{cor:cacc}}\label{s:local}
We will prove the critical \eqref{e:local-crit} and supercritical \eqref{e:local-sup} estimates in Theorem \ref{thm:local} at the same time. To this end,  let $2 \le p < \infty$ and introduce
	\begin{equation}\label{e:r} 
		r \in \left\{ 
			\begin{array}{ll} 
			\left\{ \tfrac{pp'}{2}, p \right\} & p > 2, 
			\\ (1,2) & p=2, \end{array} \right. 
		\quad q \coloneqq \left\{ 
			\begin{array}{ll} p, & p > 2, \\ 
			r, & p=2. \end{array} \right. 
	\end{equation}
The key relationship among these exponents is that 
	\begin{align}
	\label{e:sob-emb-part} &{ \norm{(Dg_1)g_2}_{L^r(\mathbb C)} \lesssim \norm{g_1}_{W^{1,p}(\mathbb C)} \norm{g_2}_{W^{1,q}(\mathbb C)}; }\\
	\label{e:sob-emb} &\norm{D(g_1g_2)}_{L^r(\mathbb C)} \lesssim \left( \|g_1\|_{L^\infty(\mathbb C)} + \|g_1\|_{W^{1,p}(\mathbb C)} \right) \|g_2\|_{W^{1,q}(\mathbb C)},\end{align}
whenever the right hand side is finite.
When $p=2$ or $r=p$, \eqref{e:sob-emb-part} is a consequence of the product rule, H\"older's inequality, and Sobolev embedding. When $p>2$ and $r=\frac{pp'}{2}$, H\"older's inequality shows that
	\[ \norm{(Dg_1)g_2}_{L^{\frac{pp'}{2}}(\mathbb C)} \le  \norm{Dg_1}_{L^{p}(\mathbb C)} \norm{g_2}_{L^{\frac{pp'}{2-p'}}(\mathbb C)}.\]
The second term is then handled by interpolation between $L^\infty$ and $L^p$, and subsequently bounded by $\norm{g_2}_{W^{1,p}(\mathbb C)}$ due to Sobolev embedding. The second estimate \eqref{e:sob-emb} follows from the product rule and applying \eqref{e:sob-emb-part}. Since the final estimate will have exponential blow-up at the endpoints of our ranges ($r$ approaching $1$ or $2$ in the case $p=2$, or in the other case as $p$ approaches $2$ or $\infty$), we use $A\lesssim B$ to denote $A \le C B$ for some $C$ depending polynomially on $k$, $p$, $r$, $q$, and {  $\|\mu\|_{W^{1,2}(\mathbb C)}$}. Most importantly, the Sobolev embedding theorems have polynomial blow up at these endpoints, hence \eqref{e:sob-emb} holds with this prescribed convention for $\lesssim$.

\subsection{Proof of Theorem \ref{thm:local}}
We begin with a few preliminary reductions. { The main step in proving Theorem \ref{thm:local} is to estimate from below	\begin{equation}\label{e:lower-bound} \norm{(I- \mu \B)g}_{W^{1,r}(\mathbb C)} \ge c(K,L,M)  \norm{g}_{W^{1,r}(\mathbb C)},\end{equation}
uniformly over $\mu$ satisfying
	\begin{equation}\label{e:mu-KLM} \norm{\mu}_{L^\infty(\mathbb C)} \le \frac{K-1}{K+1}, \quad \norm{\mu}_{W^{1,2}(\mathbb C)} \le L, \quad \norm{\mu}_{W^{1,q}(\mathbb C)} \le M,\end{equation}
with the constant $c(K,L,M)$ taking the precise form of the reciprocal on the right-hand side of \eqref{e:main-claim} below. Furthermore, by Lemma \ref{lemma:invLp}, we can focus our attention on proving \eqref{e:lower-bound} with the lower bound in the \textit{homogeneous} norm $\norm{g}_{\dot W^{1,r}(\mathbb C)}$. We can further assume in the proof of the estimate \eqref{e:lower-bound} that $\mu$ is compactly supported, thus allowing for the existence of principal solutions to the Beltrami equation. Indeed, let $\{\mu_n\}$ be a sequence of compactly supported approximations of $\mu$ such that
	\[ \norm{\mu_n}_{L^\infty(\mathbb C)} = \norm{\mu}_{L^\infty(\mathbb C)}, \quad \mu_n(x) \to \mu(x)  \ \text{a.e.} \ x \in \mathbb C, \quad \norm{ \mu_n - \mu }_{W^{1,s}(\mathbb C)} \to 0, \ s \in \{p,q\}.\]
A simple construction of such an approximation is $\mu_n = \mu \phi_n$ where $\phi_n(x) = \phi( \frac{x}{n} )$, $\phi \in \C_0^\infty(\mathbb C)$ with $\phi(0)=1$. Such $\mu_n$ also satisfies \eqref{e:mu-KLM} after inflating $L$ and $M$ by an absolute constant $A>0$. Then, assuming \eqref{e:lower-bound} has been proved for $\mu$ compactly supported,
	\begin{equation}\label{e:triangle} c(K,AL,AM) \norm{g}_{W^{1,r}(\mathbb C)} \le \norm{ (I-\mu_n \B)g}_{W^{1,r}(\mathbb C)} \le \norm{ (I-\mu \B)g}_{W^{1,r}(\mathbb C)} + \norm{ (\mu-\mu_n) \B g}_{W^{1,r}(\mathbb C)}.\end{equation}
Finally, to check that the last term in \eqref{e:triangle} goes to zero, we must verify that
	\[ \norm{ D(\mu_n-\mu) \B g }_{L^{r}(\mathbb C)} + \norm{ (\mu_n-\mu) D(\B g) }_{L^{r}(\mathbb C)} + \norm{ (\mu_n-\mu) \B g }_{L^{r}(\mathbb C)} \to 0.\]
The second and third terms indeed approach zero by the first two properties of $\mu_n$ and dominated convergence. The first term goes to zero by \eqref{e:sob-emb-part} and the third property of $\mu_n$.
}

We make the further reduction following \cite{astala01}*{pp. 39-40}. { Let $g \in \C_0^\infty(\mathbb C)$ with mean zero. Such $g$ are dense in $W^{1,r}(\mathbb C)$ hence it suffices to establish \eqref{e:lower-bound} for such $g$. By Proposition \ref{p:BC}, $\B g \in W^{n,s}(\mathbb C)$ for all $n \in \mathbb N$ and all $1<s<\infty$. Therefore, setting $h=g-\mu \B g$, since $\mu \in W^{1,q}$ we also have $h \in W^{1,q}(\mathbb C)$. Introduce $w = \mathcal K g$ and obtain that $Dw=(\B g,g)$ to see that} $w$ satisfies the inhomogeneous Beltrami equation
	\[ \overline{\partial} w = \mu \partial w + h.\]
Normalize so that $\norm{h}_{W^{1,q}(\mathbb C)}  {  =} 1$.  Then, \eqref{e:lower-bound} amounts to the \textit{a priori} estimate 
	\begin{equation}\label{e:main-claim} \|D\overline{\partial} w\|_{L^{r}(\mathbb C)} \lesssim \left\{ \begin{array}{cc} 
		\exp\left( C \max\left\{ \tfrac{1}{r-1}, \tfrac{1}{(2-r)^2} \right\} L^2 \right) & p=2; \\
		\exp\left( C \max\left\{ p,\tfrac{1}{p-2}\right\}^2 L^2 \right)\left( 1 + \Vert \mu \Vert_{W^{1,p}(\mathbb C)} \right)^2 & p>2, \end{array} \right.
	\end{equation}
where $C$ depends only on $K$. 
Indeed, the estimates in Theorem \ref{thm:local} are now special cases of \eqref{e:main-claim}. The critical case \eqref{e:local-crit} is recovered when $1<r<2$ and $p=2$; \eqref{e:local-sup} is recovered by $r=p>2$.

{ We conclude these preliminary reductions by showing how the invertibility follows from the lower bound \eqref{e:lower-bound}. First, the lower bound immediately yields the invertibility of $(I-\mu \B)$ on its range, which is indeed closed. It only remains to show the range is the whole space $W^{1,r}(\mathbb C)$, i.e. $(I-\mu \B)$ is surjective. We will use the classical method of continuity \cite[Theorem 5.2]{gt-book} applied to the operators $L_t = I- t \mu \B$ for $t \in [0,1]$. As long as $\mu$ satisfies \eqref{e:mu-KLM}, so does $t\mu$ hence we have the uniform bound
	\[ \norm{L_t g}_{W^{1,r}(\mathbb C)} \ge c(K,L,M) \norm{g}_{W^{1,r}(\mathbb C)}.\]
Furthermore, $L_0=I$ is obviously surjective, hence $L_1=I-\mu\B$ is also.}

\subsubsection{Main line of proof of Theorem \ref{thm:local}}
We proceed to establish \eqref{e:main-claim} modulo the more technical estimates \eqref{e:U1-G}-\eqref{e:V-H} below which are proved in the next subsection. 
Let $f$ be the principal solution to \eqref{eq:belt}. Introduce $u = w \circ f^{-1}$ so by the chain rule,
\begin{equation}
	\label{eq:wu} \begin{aligned} \overline{\partial} w &= \left(\partial u \circ f\right) \overline{\partial} f + \left(\overline{ \partial} u \circ f\right)\overline{ \partial f} \\
	\mu \partial w + h &= \mu\left[ \left(\partial u \circ f\right) \partial f + \left( \overline{ \partial} u \circ f\right) \overline{\overline{ \partial} f} \right]+h. \end{aligned}\end{equation}
Substituting $\mu \partial f$ for $\overline{\partial} f$ in the two instances above, and setting the two right-hand sides equal to each other, one obtains
	\begin{equation}\label{eq:h}\left (\overline{\partial} u \circ f\right) \overline{ \partial f} = \frac{h}{1-|\mu|^2}  \eqqcolon H, \quad \norm{DH}_{L^r(\mathbb C)} \lesssim
1+\|\mu\|_{W^{1,p}(\mathbb C)}.\end{equation}
where the norm estimate follows from \eqref{e:sob-emb}. In light of \eqref{eq:wu}, it remains to show $D[ (\partial u \circ f) \overline{\partial} f ]$ has $L^r$ norm controlled by the bound \eqref{e:main-claim}. By the product rule, we split $D[ (\partial u \circ f) \overline{\partial} f ]=U_1 + U_2$ where
	\[ U_1 = D\left( \partial u \circ f \right) \overline{\partial} f \quad U_2 = \left( \partial u \circ f \right) D\overline{\partial} f.\]
The difficult parts of the proof involve estimating $U_1$ and $U_2$ in terms of the following auxilliary functions
	\[  G = D( \overline{\partial} u \circ f) \overline{\partial f}, \quad \sigma = \log \partial f, \quad \lambda= D \mu + \mu D\sigma, \quad V = \left( \partial u \circ f \right) \partial f.\]
Below, we will prove
\begin{align}
	\label{e:U1-G}&\norm{U_1}_{L^r(\mathbb C)} \lesssim \norm{G}_{L^r(\mathbb C)}; \\
	\label{eq:sigma-reg} &\norm{D \sigma}_{L^p(\mathbb C)} \lesssim \exp \left( C p^2 L^2 \right) \norm{\mu}_{W^{1,p}(\mathbb C)}; \\
	\label{e:V-H} &\norm{V}_{L^s(\mathbb C)} \lesssim \exp \left( C s^2 L^2 \right) \norm{H}_{L^s(\mathbb C)}, \quad 2 \le s < \infty.
\end{align}
Assuming these for now, let us complete the proof of \eqref{e:main-claim}, and thereby Theorem \ref{thm:local}.
Comparing $G$ to $DH$, we see
	\begin{equation}\label{e:HG} DH = D\left[\left( \overline{\partial}u\circ f\right) \e^{\overline{\sigma}}\right] = \left(\overline{\partial} u\circ f\right)\overline{\partial f {\overline{D}} \sigma}+ D\left(\overline{\partial}u\circ f\right) \overline{\partial f}= H \overline{\overline{D}\sigma}+G. \end{equation}
Thus, by \eqref{e:sob-emb-part} and \eqref{eq:sigma-reg} we can provide an effective bound on $U_1$, namely
	\[ \norm{G}_{L^r(\mathbb C)} \lesssim \exp \left( C p^2 L^2 \right) \left(1+\Vert \mu \Vert_{W^{1,p}(\mathbb C)}\right). \]
So as a consequence of \eqref{e:U1-G} and \eqref{eq:sigma-reg}, we have 
\begin{equation}\label{e:1B}\norm{D \overline{\partial}w}_{L^r(\mathbb C)} \lesssim \exp \left( C \max\left\{p^2, \tfrac{1}{r-1} \right\} L^2 \right)\left(1+\|\mu\|_{W^{1,p}(\mathbb C)}\right) + \norm{ U_2 }_{L^r(\mathbb C)}.\end{equation}
For $U_2$, we notice that by replacing $\overline{\partial} f$ with $\mu \partial f$ we can write 
	\begin{equation}\label{e:U2Vlambda} U_2 = \left( \partial u \circ f \right) D(\mu \partial f) = \left( \partial u \circ f \right) \left[ (D\mu) \partial f + \mu D \partial f \right] = V \lambda.\end{equation}
We will apply \eqref{e:V-H} to control $U_2$. To this end, we first assume $r<p$ and let $s= \frac{pr}{p-r}$. Recalling the definition of $H$ in \eqref{eq:h}, applying the Sobolev embedding, and using \eqref{eq:sigma-reg}, we have
	\begin{equation}\label{e:Hlambda} \norm{H}_{L^s(\mathbb C)} \lesssim \norm{h}_{L^s(\mathbb C)} \lesssim 1, \quad \norm{\lambda}_{L^p(\mathbb C)} \lesssim \exp \left( C p^2 L^2 \right)\left( 1 + \norm{\mu}_{W^{1,p}(\mathbb C)} \right). \end{equation}
We can now apply by H\"older's inequality to obtain 
	\begin{equation}\label{e:U2}\norm{U_2}_{L^r(\mathbb C)} \le \norm{V}_{L^s(\mathbb C)} \norm{\lambda}_{L^p(\mathbb C)} \lesssim \exp \left[ C (s^2+p^2) L^2 \right]\left( 1 + \norm{\mu}_{W^{1,p}(\mathbb C)} \right).
\end{equation}
When $p=2$, $r$ is always strictly less than $p$. Furthermore, $s \le \frac{4}{2-r}$, so combining \eqref{e:1B} and \eqref{e:U2} concludes the proof of \eqref{e:main-claim} for $p=2$. We now restrict to $p>2$, in which case $r = \frac{pp'}{2}$ and $s = \frac{p^2}{p-2} \le 3 \max\left\{ p, \tfrac{1}{p-2} \right\}$. In this case \eqref{e:1B} and \eqref{e:U2} imply
	\begin{equation}\label{e:summ1} \norm{\overline{\partial} w}_{\dot W^{1,\frac{pp'}{2}}(\mathbb C)} \lesssim \exp \left( C \max \left \{ p,\tfrac{1}{p-2} \right\}^2 L^2 \right) \left(1+ \|\mu\|_{W^{1,p}(\mathbb C)}\right). \end{equation}
However, by Lemma \ref{lemma:invLp}, $\norm{\overline{\partial} w}_{L^{\frac {pp'}{2}}(\mathbb C)}$ is also controlled by the right hand side of \eqref{e:summ1}. Moreover, since $\partial w = \B(\overline{\partial}w)$ and $\B$ is bounded on $W^{n,s}(\mathbb C)$ by virtue of Proposition \ref{p:BC}, in fact \eqref{e:summ1} holds for the larger nonhomogeneous norm $\norm{D w}_{W^{1,\frac{pp'}{2}}(\mathbb C)}$. Since $p>2$, one can check that $\frac {pp'}{2} > 2$. Therefore the Sobolev embedding shows that $\partial w$ is in fact in $L^\infty(\mathbb C)$ with norm controlled by \eqref{e:summ1}. This will allow us to take $s=\infty$ in \eqref{e:U2} by the identity
	\begin{equation}\label{e:wV} \partial w = \left(\partial u \circ f\right) \partial f + \left( \overline{ \partial} u \circ f\right) \overline{\overline{ \partial} f} =  \left(\partial u \circ f\right) \partial f + \left( \overline{ \partial} u \circ f\right) \overline{\mu \partial f}=V +\overline{\mu} H.\end{equation}
The first equality above is simply the chain rule and second uses the fact $\overline{\partial } f = \mu \partial f$. Therefore the $L^\infty(\mathbb C)$ norm of $V$ also obeys the bound \eqref{e:summ1}. Therefore,
	\[ \norm{U_2}_{L^p(\mathbb C)} \le \norm{V}_{L^\infty(\mathbb C)} \norm{\lambda}_{L^p(\mathbb C)} \lesssim \exp \left( C \max \left \{ p,\tfrac{1}{p-2} \right\}^2 L^2 \right) \left(1+ \|\mu\|_{W^{1,p}(\mathbb C)}\right)^2.\]
Applying the above estimate to \eqref{e:1B} with $r=p$ establishes \eqref{e:main-claim} for $p > 2$.

\subsubsection{Proofs of \eqref{e:U1-G}-\eqref{e:V-H}}\label{ss:proof-technical}
\begin{proof}[Proof of \eqref{eq:sigma-reg}]
By \cite{astala-book}*{Theorem 5.2.3} for $\mu \in W^{1,2}(\mathbb C)$, the principal solution $f$ can be written as $\partial f = \e^\sigma$ where $\sigma$ satisfies 
	\begin{equation}\label{eq:sigma} \overline{\partial} \sigma = \mu \partial \sigma + \partial \mu.\end{equation} Therefore, $(I-\mu \B)\partial \sigma=\partial \mu$ so by Lemma \ref{lemma:invLp} we obtain \eqref{eq:sigma-reg}.
\end{proof}

\begin{proof}[Proof of \eqref{e:V-H}]
By change of variable and the quasiconformality of $f$ in the form \eqref{e:Jf-quasi} we have
\begin{equation}\label{e:2B}  \int_{\mathbb C} \abs{V}^s =\int_{\mathbb C} \vert \partial u\circ f\vert^{s} \vert \partial f\vert^s 
 = \int_{\mathbb C} \vert \partial u\vert^{s} \vert \partial f \circ f^{-1} \vert^s  \vert Jf^{-1} \vert\lesssim \int_{\mathbb C} \vert \partial u\vert^{s} \vert Jf^{-1} \vert^{1-\frac{s}{2}}. \end{equation}
We connect back to $H=\left (\overline{\partial} u \circ f\right) \overline{ \partial f}$ through the weighted Lebesgue estimates for $\B$ contained in Proposition \ref{p:BC} with the weight $\omega=\vert Jf^{-1} \vert^{1-\frac{s}{2}}$. Recalling the estimate of $[\omega]_{\A_p(\mathbb C)}$ from Lemma \ref{lemma:moser}, we see
	\[ \norm{\partial u}_{L^s(\omega)}  \lesssim \exp \left( C s^2 L^2 \right) \norm{\overline{\partial} u}_{L^s(\omega)}.\]
Changing variables back and using $|Jf^{-1} \circ f|^{-s/2} \le |\partial f|^s$ establishes \eqref{e:V-H}.
\end{proof}

\begin{proof}[Proof of \eqref{e:U1-G}]
We follow a similar outline to the proof of \eqref{e:V-H} that was just given. First, by the chain rule and \eqref{e:Jf-quasi},
	\[ \abs{U_1} = \abs{ D(\partial u \circ f) \overline{\partial} f } 
    \lesssim \abs{(D\partial u \circ f)} \left(\abs{ \overline{\partial} f } + \abs{ \partial f } \right) \abs{ \overline{\partial} f } \lesssim \abs{(D\partial u \circ f)} \abs{ Jf }.\]
Changing variables and introducing $\B$ according to \eqref{e:twine} we have
	\begin{equation}\label{e:SobBeurling}  \left\|U_1 \right\|_{L^r(\mathbb C)} \lesssim \norm{D \B(\overline{\partial}u)}_{L^r(\omega)}, \quad \omega = \abs{Jf^{-1}}^{1-r}. \end{equation}
	By Lemma \ref{lemma:moser}, $|Jf^{-1}|^{1-r} \in \A_r(\mathbb C)$ with the dependence \eqref{e:Jf-Ap-2} so the weighted Sobolev estimate for $\B$ (Proposition \eqref{p:BC} with $n=1$ and $p=r$) implies
	\begin{equation}\label{e:1A} \begin{aligned} \norm{ U_1 }_{L^r(\mathbb C)} &\lesssim \exp \left( C \max\left\{r^2, \tfrac{1}{r-1} \right\} L^2 \right) \left( \int_{\mathbb C} |D\overline{\partial} u |^r |J{f^{-1}}|^{1-r} \right)^{\frac 1r} \\	&\lesssim \exp \left( C \max\left\{p^2, \tfrac{1}{r-1} \right\} L^2 \right)\left( \int_{\mathbb C} |D\overline{\partial} u \circ f|^r |Jf|^{r} \right)^{\frac 1r}, \end{aligned}\end{equation}
where the last line followed from changing variables back and the fact that $r \le p$. The proof will be concluded by the pointwise estimate
\begin{equation}\label{e:pointG} 
	\abs{ J f  } \abs{D \overline{\partial} u \circ f } \lesssim \abs{G},
\end{equation}
which we now proceed to establish. First, the chain rule yields $D \left( \overline{\partial} u \circ f \right) = Jf \left( D\overline{\partial} u \circ f \right)$. Inverting $Jf$ and multiplying by $\abs{Jf}$ yields
	\begin{equation}\label{e:compG} \abs{Jf} \left( D\overline{\partial} u \circ f \right) = \abs{Jf} (Jf)^{-1} D \left( \overline{\partial} u \circ f \right).
\end{equation}
Thus, upon precisely computing
	\[ (Jf)^{-1} = \frac{1}{\abs{J f}} \left[ \begin{array}{cc} 
 \overline{ \partial f}   & - \overline{ \overline{\partial} f } \\ - \overline{\partial} f & \partial f 
\end{array}\right] \]
we obtain from quasiconformality that pointwise, the norm of the matrix $\abs{Jf}(Jf)^{-1}$ is controlled by $\abs{\partial f}$. Then recalling that $G = D( \partial u \circ f) \overline{\partial f}$, the desired estimate \eqref{e:pointG} is a consequence of \eqref{e:compG}.
\end{proof}

\subsection{Proof of Corollary \ref{cor:cacc}}\label{ss:cacc}
We will again give a unified proof of the critical and supercritical cases. To this end, let $p \ge 2$ and consider
	\[ q \in \left\{ \begin{array}{cc} (2,\infty), & p=2, \\ \left[ \frac{p}{p-1},\infty\right) & p > 2, \end{array} \right. \quad r =\frac{q}{q-1}.\]
Then, \eqref{eq:cacc-1} and \eqref{eq:cacc-1-sup} {  are special cases of}
	\begin{equation}\label{eq:cacc-1-un}\|\eta (Df)\|_{L^q(\mathbb C)} \lesssim \|(D\eta)f\|_{L^q(\mathbb C)}, \end{equation}
while \eqref{eq:cacc-2} and \eqref{eq:cacc-2-sup} {  are special cases of} 
	\begin{equation}\label{eq:cacc-2-un} \|\eta (D^2 f)\|_{L^r(\mathbb C)} \lesssim \|(D \eta)f\|_{L^r(\mathbb C)} +  \|(D \eta)(Df)\|_{L^r(\mathbb C)} +  \|(D^2\eta)f\|_{L^r(\mathbb C)}. \end{equation}
In this proof, we will not track the constants implicit in \eqref{eq:cacc-1-un} and \eqref{eq:cacc-2-un} but they are an absolute constant multiple of the corresponding estimates for $(I - \mu \B)^{-1}$ in Lemma \ref{lemma:invLp} and Theorem \ref{thm:local}.
First we will prove \eqref{eq:cacc-1-un}. Since $f \in L^q_{\mathrm{loc}}$, we claim that we can extend \eqref{e:BD} to all $\psi$ of the form $\eta \phi$ for $\eta \in \C_0^\infty({ \Omega})$ and $\phi \in W^{1,r}(\mathbb C)$.
{ Let us introduce the shorthand $\H_\mu$ for the distributional Beltrami derivative $\overline{\partial}-\mu \partial$, and $\H_\mu^*$ its formal adjoint, initially defined on functions $\varphi \in \C^\infty_0(\Omega)$ by
	\[ -\H_\mu^*\varphi = \overline \partial \varphi - \partial(\mu \varphi).\]}
So, for such $\eta$ and $\phi$, by H\"older's inequality and Sobolev embedding,
	\begin{equation}\label{e:Hmuest} \|{ \H_\mu^*}(\eta \phi)\|_{L^{r}(\mathbb C)}
	\lesssim \norm{\eta \phi}_{W^{1,r}(\mathbb C)}+\left(\norm{\eta\mu}_{L^\infty(\mathbb C)} + \norm{\eta\mu}_{W^{1,p}(\mathbb C)}\right) \norm{\phi}_{W^{1,r}(\mathbb C)}.\end{equation}
Therefore, \eqref{e:BD} holds for all such $\psi =\eta \phi$ by density.
Now, fix $\eta \in \C_0^\infty({ \Omega})$, let $\chi \in \C_0^\infty({ \Omega})$ with $\chi \equiv 1$ on $\supp \eta$, and set $\nu = \mu \chi$. Let $g \in C^\infty_0({ \Omega})$ and let $v$ satisfy
	\begin{equation}\label{eq:belt-c} \overline{\partial}v - \nu \partial v = g.\end{equation}
Then, Lemma \ref{lemma:invLp} and Theorem \ref{thm:local} respectively imply
	\begin{equation}\label{eq:sob-est} { \|Dv\|_{L^{r}(\mathbb C)}} \lesssim \|g\|_{L^{r}(\mathbb C)}, \quad { \|Dv\|_{W^{1,r}(\mathbb C)}} \lesssim \|g\|_{W^{1,r}(\mathbb C)}
	. \end{equation}
Furthermore, applying $\partial$ to \eqref{eq:belt-c} we obtain
	\[ \partial g = {- \H_\nu^*}(\partial v).\]
{ The same argument used to show \eqref{e:Hmuest} also shows that $\H_\nu^*$ can continuously be extended to $W^{1,r}(\mathbb C)$ and so the above display is well-defined.}
Pair the above display with $F=\eta f$ to obtain
	\[ \left\l F,\partial g \right\r = -\left\l \eta f,\H_\nu^* \partial v \right\r = -\left\l f,\H_\nu^*(\eta \partial v) \right\r + \left\l f, \left( \overline{\partial}\eta - \nu \partial \eta \right)\partial v \right\r.\] 
The precise form of $\chi$ shows that $\nu \eta = \mu \chi \eta = \mu \eta$ so that, since $\partial v \in W^{1,r}(\mathbb C)$ by \eqref{eq:sob-est}, the {  first term in the right hand side} in the above display vanishes by \eqref{e:BD}. On the other hand, the first estimate in \eqref{eq:sob-est} shows that
	\[\abs{\left\l f, \left( \overline{\partial}\eta - \nu \partial \eta \right)\partial v \right\r} \lesssim \|(D\eta) f\|_{L^q(\mathbb C)} \|g\|_{L^r(\mathbb C)}.\]
Therefore, combining the previous two displays and using compact support,
	\[ |\l \partial F, g \r| \lesssim \|(D \eta) f\|_{L^q(\mathbb C)} \|g\|_{L^r(\mathbb C)},\]
for all $g \in C^\infty_0$. This establishes that $\norm{ \eta \partial f }_{L^q}$ is bounded by the right hand side of \eqref{eq:cacc-1-un}. To prove the same for $\eta \overline{\partial}f$, simply notice that by \eqref{e:BD}, there holds for any $h \in \C^{\infty}_0({ \Omega})$,
	\[ \left| \left\l \eta \overline{\partial}f,h \right \r \right| \le \left| \left\l \mu\eta \partial f,h \right \r \right|. \]

We proceed to prove \eqref{eq:cacc-2-un}. Notice that \eqref{eq:cacc-1-un}, which we just proved, together with the assumption $f \in L^q_{\mathrm{loc}}$ { establishes that $f \in W^{1,q}_{\mathrm{loc}}$. Hence the equality $\H_\mu f=0$ can be upgraded from holding distributionally to holding almost everywhere}. Furthermore, $F=\eta f$ indeed belongs to $W^{1,s}(\mathbb C)$ for every $s \le q$ and
	\[ \overline{\partial} F -\nu \partial F= (\overline{\partial}\eta-\nu\partial\eta)f.\]
Using Sobolev embedding, the right hand side of the above display belongs to $W^{1,r}$ with norm
	\begin{equation}\label{eq:eta-mu-phi-s} \| (\overline{\partial}\eta-\nu\partial\eta)f\|_{W^{1,r}(\mathbb C)} \lesssim \|(D \eta)f\|_{L^r(\mathbb C)} + \|(D\eta)(Df)\|_{L^r(\mathbb C)} + \|(D^2\eta)f\|_{L^r(\mathbb C)}.\end{equation} 
On the other hand, by \eqref{e:twine} and Proposition \ref{p:BC}, 
	\[ \|D^2F\|_{L^{r}(\mathbb C)} \le \left\Vert\overline{\partial}F\right\Vert_{W^{1,r}(\mathbb C)} + \|\partial F\|_{W^{1,r}(\mathbb C)} \lesssim \left\Vert\overline{\partial}F\right\Vert_{W^{1,r}(\mathbb C)}.\]
Furthermore, Theorem \ref{thm:local} applies to give
	\[ \norm{ \overline{\partial} F }_{W^{1,r}(\mathbb C)} \lesssim \| (\overline{\partial}\eta-\nu\partial\eta)f\|_{W^{1,r}(\mathbb C)} \]
and the proof of \eqref{eq:cacc-2-un} is concluded by \eqref{eq:eta-mu-phi-s}.

Finally, we establish the concluding statement that $f \in W^{2,r}_{\mathrm{loc}}({ \Omega})$. To this end, we iterate the first Caccioppoli inequality \eqref{eq:cacc-1-un} to show that $f \in W^{1,s}_{\mathrm{loc}}({ \Omega})$ for every $1<s<\infty$. Indeed, let $f \in L^{q}_{\mathrm{loc}}({ \Omega})$ and suppose $s>q>1$. By \eqref{eq:cacc-1-un}, $f \in W^{1,q}_{\mathrm{loc}}({ \Omega})$ so by Sobolev embedding, $f \in L^{q_1}_{\mathrm{loc}}(E)$ for every $q_1$ satisfying 
	\[ q_1 < \left\{ \begin{array}{cc} \frac{2q}{2-q}, & q \le  2; \\ \infty, & q > 2. \end{array} \right. \]
In particular, we can choose $q_1>2$ so that a second application of \eqref{eq:cacc-1-un} implies $f \in W^{1,q_1}_{\mathrm{loc}}({ \Omega})$ and by Sobolev embedding $f \in L^{q_2}_{\mathrm{loc}}({ \Omega})$ for every $q_2 < \infty$. One final application of \eqref{eq:cacc-1-un} proves the claim. In particular, $f \in W^{1,r}_{\mathrm{loc}}({ \Omega})$ so the right hand side of \eqref{eq:cacc-2-un} is finite hence $f \in W^{2,r}_{\mathrm{loc}}({ \Omega})$.

{ 
\subsubsection{Remark on the bilipschitz nature of principal solutions}\label{ss:bilip}
Recall the two complementary formulas for principal solutions $f$ satisfying $\overline{\partial} f = \mu \partial f$, from \eqref{e:principal-rep-1} and the proof of \eqref{eq:sigma-reg}. Let us suppose $\mu$ is supported in a compact set $K$. Without resorting to the full strength of Theorem \ref{thm:local}, we can obtain the following crude bilipschitz estimate for $f$. Recall that $\partial f= \e^{\sigma}$ where $\sigma$ satisfies 
	\[ \overline{\partial} \sigma = \mu \partial \sigma + \partial \mu.\]
According to \eqref{eq:sigma-reg}, we conclude that $\norm{ D \sigma}_{L^p(\mathbb C)} \lesssim \norm{\mu}_{W^{1,p}(\mathbb C)}$. Furthermore, since $\mu$ is supported in $K$, so is $\overline{\partial} \sigma$. Therefore, $\sigma=\mathcal K( \overline{\partial} \sigma )$, and by standard $L^p$ estimates for $\mathcal K$ \cite[Theorem 4.3.11]{astala-book}, we infer
	\begin{equation}\label{e:sigma-infty}\norm{\sigma}_{L^\infty(\mathbb C)} \lesssim \sqrt{ \norm{\overline{\partial}\sigma}_{L^{\frac p{p-1}}(K)} \norm{\overline{\partial} \sigma}_{L^p(K)} } \lesssim |K|^{\frac{p-2}{2p}} \norm{\mu}_{W^{1,p}(\mathbb C)}, \end{equation}
which provides upper and lower bounds on $\abs{Df}$ which depend exponentially on the bound in \eqref{e:sigma-infty}. Applying Theorem \ref{thm:local} and \eqref{e:principal-rep-1}, we obtain an improvement of this \textit{upper} bound which depends polynomially on $\norm{\mu}_{W^{1,p}(\mathbb C)}$ and is in fact independent of the size of $K$. It would be interesting to see if our method could also be used to obtain an improvement on the lower bound, but at this time we do not see how to do so.
}

\section{Proof of Theorem \ref{thm:global} and Corollary \ref{c:global}}\label{s:global}
Throughout, let $2 < p <\infty$ and $\Omega$ be a simply connected bounded domain in $\mathbb C$. The restriction to simply connected domains is for convenience, but can be lifted to finitely connected domains. In this section we still consider global solutions to the Beltrami equation, but we assume $\mu$ to be of a special form, linked to $\Omega$ in the following way. Let $\mu$ satisfy
	\begin{equation}\label{eq:mu-klO} \supp \mu \subset \overline{\Omega}, \quad \left\Vert \mu\right\Vert_\infty=k=\frac{K-1}{K+1} < 1, \quad \mu \in W^{1,p}(\Omega).\end{equation}
In proving Theorem \ref{thm:global}, the Beurling-Ahlfors operator $\B$ is replaced by its compression to a domain $O$, defined by
	\[ \B_O = \mathbf 1_{ \overline{O}} \B( \cdot \mathbf 1_{\overline{O}} ).\]
Lebesgue space estimates for $\B_O$ follow from those established for $\B$ in Proposition \ref{p:BC}, by simply {  considering} functions supported in $\overline{O}$. However, the Sobolev estimates must be approached differently. The systematic study of unweighted Sobolev estimates for compressions of Calder\'on-Zygmund operators was taken up by Prats and Tolsa in \cite{prats-tolsa}.  A completely different approach by the authors has led to general $T(1)$ theorems and weighted estimates in Sobolev spaces in \cite{diplinio23domains}. The necessary ingredients from these papers are extracted in Lemma \ref{lemma:Bp} below.

\begin{definition}\label{def:Bp} We first define two norms for functions $f: \Gamma \to \mathbb C$ where $\Gamma$ is a piecewise continuous curve. Say $f$ is Dini continuous if the norm
	\[ \norm{f}_{\mathrm{Dini}} = \int_0^1 \sup_{|x-y|\le t} \abs{f(x)-f(y)} \, \frac {\d t}{t} \]
is finite. Define the homogeneous Besov norm on $\Gamma$ by
	\[ \norm{f}_{\dot B^{1-\frac 1q}_{q,q}(\Gamma)} = \left( \int_{\Gamma}\int_{\Gamma} \frac{ |f(x)-f(y)|^q}{|x-y|^q} \, \d s(x) \, \d s(y) \right)^{\frac 1q} < \infty,\]
where $\d s$ is the surface measure on $\Gamma$. {  Following  \cite{astala-prats-saksman}, we say a bounded simply connected domain $O$ is a Dini-smooth domain if there exists a bilipschitz parameterization $A : \mathbb T \to \partial \Omega$ such that $\norm{A'}_{\mathrm{Dini}} < \infty$.
Dini-smooth domains are the natural setting for higher order conformal estimates. In particular, we can take the canonical parameterization to be the trace of Riemann map of $\Omega$; see Lemma \ref{lemma:conf} below. Given $q> 2$, we say a Dini-smooth domain $O$ is a $\Bes_q$ domain if there exists a bilipschitz $A:\mathbb T \to \partial \Omega$ such that
	\[ \norm{A'}_{\dot B^{1-\frac 1q}_{q,q}(\mathbb T)} <\infty.\]
Letting $N_O$ denote the normal vector to the boundary $\partial O$, there holds
	\begin{equation}\label{e:param-equiv} \norm{N_O}_{\dot B^{1-\frac 1q}_{q,q}(\partial O)} \lesssim \norm{A'}_{\dot B^{1-\frac 1q}_{q,q}(\mathbb T)} \lesssim \norm{N_O}_{\dot B^{1-\frac 1q}_{q,q}(\partial O)},\end{equation}}
with implicit constants in \eqref{e:param-equiv} depending only on the bilpschitz character of $A$; see \cite{cruz-tolsa}*{Lemmata 3.1 and 3.3}. In light of \eqref{e:param-equiv}, we define 
	\[ \norm{O}_{\Bes_q} = \norm{N_O}_{\dot B^{1-\frac 1q}_{q,q}(\partial O)}. \]
\end{definition}

\subsection{Standing assumptions and implicit constants}\label{ss:assumptions}
Let us fix some parameters and establish some notational conventions for the remainder of this section. Henceforth, let $2<r \le p$, $\Omega$ be a bounded simply connected $\Bes_p$ domain, and $\mu$ satisfy \eqref{eq:mu-klO}. Furthermore, let $f$ be the principal solution of \eqref{eq:belt}, set $O=f(\Omega)$.

The shorthand $\lesssim$ and $\sim$ will denote one or two-sided inequalities with implicit dependence on { lower-order terms, generically understood. The lowest order terms are $[\Omega]_{\mathrm{Dini}}$, $[O]_{\mathrm{Dini}}$, $\diam \Omega$, $\diam O$, $p$, $r$, and $K$; these are nearly always omitted. In general, we only track the dependence on the highest order terms $\norm{\Omega}_{\Bes_p}$, $\norm{f(\Omega)}_{\Bes_p}$ and $\norm{\mu}_{W^{1,p}(\Omega)}$, though lower order terms such as $\norm{\mu}_{W^{1,r}(\Omega)}$, $\norm{\Omega}_{\Bes_2}$ and $\norm{f(\Omega)}_{\Bes_2}$ may be present from time to time for clarity.}  
Furthermore, $\mathcal G$ and $\mathcal E$ will be generic functions such that $\mathcal G$ and $\log \mathcal E$ have polynomial growth, which are implicitly determined by lower order parameters.

\subsection{Weighted estimates for compressions}
To apply the same strategy as in the proof of Theorem \ref{thm:local}, we will need estimates for $\B_O$ on a certain weighted Sobolev space $W^{1,p}(O,\omega)$. It is known that the Besov norm of the boundary normal introduced above is precisely connected to \textit{unweighted} Sobolev estimates for $\B_O$; consult the following references \cites{cruz-tolsa,tolsa2013regularity,prats17}. In particular,
	\begin{equation}\label{e:NO} \left\Vert \B_O \right\Vert_{\mathcal L(W^{1,p}(O))} \sim 1 + \left\Vert O \right \Vert_{\Bes_p}, \quad p>2. \end{equation}
An analog of the quantitative geometric characterization \eqref{e:NO} for $\norm{ \B_O}_{\mathcal L(W^{1,p}(O,\omega))}$ is not currently available, though qualitatively they are equivalent. In fact, the following Lemma demonstrates the sharpness of our assumption on $O=f(\Omega)$ in Theorem \ref{thm:global}.

\begin{lemma}\label{lemma:sharp}
If $I-\mu \B_\Omega$ is invertible on $W^{1,p}(\Omega)$, then
	\[ \B_O:W^{1,p}(O,\omega) \to W^{1,p}(O,\omega), \quad \omega = \abs{Jf^{-1}}^{1-p}.\]
\end{lemma}

The proof is postponed until \S \ref{ss:sharp} below. Concerning quantitative geometric conditions on $O$, using the $T(1)$ theorems developed by the authors in \cite{diplinio23domains}, we can give some alternatives of varying of degrees of sharpness.

\begin{lemma}\label{lemma:Bp} 
Set $\omega = \left\vert Jf^{-1} \right\vert^{1-p}$. Then, 
\begin{itemize} 
	\item[i.] $\norm{ \B_O}_{\mathcal L(W^{1,p}(O,\omega))} \lesssim \left(\|Jf\|_\infty \|Jf^{-1} \|_\infty \right)^{1-\frac{1}{p}} \left( 1+ \left\Vert O \right\Vert_{\Bes_p} \right)$.
	\item[ii.] $\norm{ \B_O}_{\mathcal L(W^{1,p}(O,\omega))} \lesssim  1 + \frac{\left\Vert \B_O(\mathbf 1) \right\Vert_{W^{1,p}(O,\omega)}}{ \norm{\mathbf 1_O}_{L^p(O,\omega)} } $.
	\item[iii.] For any $\ep>0$, $\norm{ \B_O}_{\mathcal L(W^{1,p}(O,\omega))} \lesssim  \left( 1 + \left\Vert O \right\Vert_{\Bes_{p+\ep}}\right)\mathcal E \left(\ep^{-1} \right)$.
\end{itemize}
\end{lemma}

\begin{remark}\label{rem:unweight}
Lemma \ref{lemma:Bp}.i is achieved by appealing to the unweighted estimate \eqref{e:NO}, and  crude bounds on $\norm{\abs{Jf}^{-1}}_\infty$ and $\norm{Jf}_\infty$ can be derived from { \eqref{e:Fp} of} Proposition \ref{prop:principal}. In this way, if the reader is interested in a quick bound for $(I - \mu \B_\Omega)^{-1}$ which bypasses the more difficult weighted Sobolev theory of CZOs that ii. and iii. rely on, one can appeal to the { logarithmic Sobolev  inequality}
	\[ \norm{g}_\infty \lesssim \frac{1}{q-2} \left( 1 + \norm{g}_{W^{1,2}(\Omega)} \right) \log \left( \e + \norm{g}_{W^{1,q}(\Omega)} \right), \quad q > 2,\]
to obtain
	\[ \norm{ \B_O}_{\mathcal L(W^{1,p}(O,\omega))} \lesssim \left( 1+ \norm{O}_{\Bes_q} + \norm{\Omega}_{\Bes_q} + \norm{\mu}_{W^{1,q}(\Omega)} \right)^{\frac{C}{q-2}} \left\Vert O \right\Vert_{\Bes_p}.\]
Notice that the middle factor blows up exponentially as $q \to 2$, while weighted estimates in e.g.\ Lemma \ref{lemma:Bp}.iii and Corollary \ref{c:global} produce an absolute polynomial bound in terms of the boundary data $\norm{\Omega}_{\Bes_p}$ and $\norm{O}_{\Bes_p}$.
\end{remark}

To prove Lemma \ref{lemma:Bp}, we need to verify that the weights $\omega=\abs{Jf^{-1}}^{1-p}$ belong to the appropriate Muckenhoupt $\A_p$ class for the results from \cite{diplinio23domains} to apply. The following adaptation of Lemma \ref{lemma:moser} is proved in \S \ref{ss:moser2}.

\begin{lemma}\label{lemma:moser2}
Let $F \in W^{1,2}_{\mathrm{loc}}(\Omega)$ be a homeomorphism satisfying
	\[ \overline{\partial} F(z) = \mu(z) \partial F(z), \quad z \in \Omega,\]
and further assume that $F(\Omega)$ is a bounded $\Bes_r$ domain. Then, for each $a \in \mathbb R$ and $1<q<\infty$, there exists $\upsilon:\mathbb C \to [0,\infty]$ such that 
	\[ \upsilon = \left \vert JF^{-1} \right \vert^a \mbox{ on } F(\Omega), \]
and
	\begin{equation}\label{e:upsAq} \left[\upsilon\right]_{\A_q(\mathbb C)} \lesssim\mathcal E\left (|a|, \tfrac {1}{q-1}\right). 
	\end{equation}
\end{lemma}

\begin{proof}[Proof of Lemma \ref{lemma:Bp}]
i. follows by pulling the weight $\omega$ outside, applying the unweighted estimate, using \eqref{e:NO}, and then reinserting the weight. To prove ii. and iii., we rely on the weighted Sobolev estimates for Calder\'on-Zygmund operators on domains established in \cite{diplinio23domains}. { In fact, the first statement in \cite[Theorem C]{diplinio23domains} directly implies ii. once we verify $\omega$ belongs to the weight class $\A_{\frac r2}(O)$ with characteristic $[\omega]_{\A_{\frac r2}(O   )}$ controlled by lower-order terms
	\[ [v]_{\A_t(O   )} \coloneqq \sup_{\substack{Q \ \mathrm{cube}\\   |Q \cap O   | >0}} \langle \mathbf 1_O    v \rangle_{Q} \left\langle \mathbf 1_O    v^{-1} \right\rangle_{\frac{1}{t-1},Q}, \]
which coincides with \eqref{def:Ap} when $O    = \mathbb C$. Therefore, letting $\upsilon$ be the extension of $\omega$ provided by Lemma \ref{lemma:moser2} above, this follows from the trivial observation $[\omega]_{\A_{t}(O   )} \le [\upsilon]_{\A_{t}(\mathbb C)}$ and \eqref{e:upsAq}. We will derive iii. from ii. plus a second application of Lemma \ref{lemma:moser2}. Let $\ep>0$ and set $s = \frac{p+\ep}{\ep}$ and take $\tilde\upsilon$ to be the extension of $\omega^s$ provided by Lemma \ref{lemma:moser2}. Let $Q$ be a large cube so that $O \subset Q$ and $|O| \sim |Q|$. Then, for $q=1+s$, we note that $\frac{-s}{q-1}=-1$ so that
	\[ \avg{\mathbf 1_O    \omega^s}_{Q} \le [\tilde \upsilon]_{\A_q(\mathbb C)} \avg{\mathbf 1_O    \omega^{-1}}_Q^{s}.\]
Combining the above display with the trivial consequence of H\"older's inequality, 
	\[ 1 \lesssim \left( \frac{|O|}{|Q|} \right)^2 \lesssim \avg{ \mathbf 1_O    \omega}_Q \avg{\mathbf 1_O \omega^{-1}}_Q,\]
we obtain $\norm{\omega^s}_{L^1(O)} \lesssim [\tilde \upsilon]_{\A_q(\mathbb C)} \norm{\omega}_{L^1(O)}^s$. Therefore, by Ho\"lder's inequality and \eqref{e:NO},
	\[ \norm{\B_O(\mathbf 1) }_{W^{1,p}(O,\omega)} \le \norm{\B_O(\mathbf 1)}_{W^{1,p+\ep}(O)}\norm{\omega^s}_{L^{1}(O)}^{\frac{1}{sp}} \lesssim \left( 1 + \norm{O}_{{  \Bes_{p+\ep}}} \right) \mathcal [\tilde\upsilon]_{\A_q(\mathbb C)}^{\frac{1}{sp}} \norm{\mathbf 1}_{L^p(O,\omega)}. \]
Finally, from \eqref{e:upsAq}, the characteristic of $\tilde \upsilon$ is controlled by $\mathcal E(s)$ and $s \lesssim \ep^{-1}$.
}
\end{proof}

The second consequence of Lemma \ref{lemma:moser2} is the following analogue of Lemma \ref{lemma:invLp}.

\begin{proposition}\label{prop:LpBO}
For each $1<s<\infty$,
	\begin{equation}\label{e:LpBO} \left \Vert (I - \mu\B_\Omega)^{-1} \right \Vert_{\mathcal L(L^s(\Omega))} \lesssim\mathcal E\left(s, \tfrac{1}{s-1} \right).
\end{equation}
\end{proposition}
\begin{proof}
The proof consists of one small modification to the argument leading to \eqref{e:AIS} which we outline. 
Arguing along the lines of the proof of Theorem \ref{thm:local}, it suffices to show $\|\overline{\partial}w\|_{L^s(\Omega)} \lesssim \|h\|_{L^s(\Omega)}$ with the implicit constant controlled as in \eqref{e:LpBO}, where $w$ satisfies
 	\[ \overline{\partial}w = \mu \partial w + h .\]
{ Then one can conclude the invertibility either by the method of continuity as in Theorem \ref{thm:local} or by the method of \cite[Theorem 1]{astala01}.}
Let $f$ be $\mu$-quasiconformal and set $u=w \circ f^{-1}$. The chain rule shows that $\overline{\partial}w =  (\partial u \circ f) \overline{\partial} f + (\overline{\partial} u \circ f)\overline{ \partial f}$ and the equations for $w$ and $f$ imply $(1-|\mu|^2)^{-1} h = (\overline{\partial} u \circ f)\overline{ \partial f}$. Therefore it remains to estimate $(\partial u \circ f) \overline{\partial} f$ by $h$ in $L^s(\Omega)$-norm. Changing variables, using \eqref{e:Jf-quasi} and \eqref{e:twine},
	\[ \int_\Omega \left \vert (\partial u \circ f) \bar\partial f \right \vert^s \lesssim \int_{f(\Omega)} \left \vert \B(\overline{\partial} u) \right \vert^s \left \vert Jf^{-1} \right \vert^{1-s/2}.\]
Letting $\upsilon$ be the extension of $\abs{Jf^{-1}}^{1- \frac s2}$ from Lemma \ref{lemma:moser2}, and applying Proposition \ref{p:BC},
	\[  \int_{f(\Omega)} \left \vert \B(\overline{\partial} u) \right \vert^s \left \vert Jf^{-1} \right \vert^{1-s/2} \le \left \Vert \B(\overline{\partial} u) \right \Vert_{L^s(\mathbb C,\upsilon)}^s \lesssim { [\upsilon]_{\A_s(\mathbb C)}^{\max\{s,\frac{s}{s-1}\}} } \left \Vert \overline{\partial} u \right \Vert_{L^s(\mathbb C,\upsilon)}^s.\]
The proof is concluded by estimating $[\upsilon]_{\A_s(\mathbb C)}$ by \eqref{e:upsAq} and changing variables back to obtain $\left \Vert \overline{\partial} u \right \Vert_{L^s(\mathbb C,\upsilon)}{  \lesssim}\|h\|_{L^s(\Omega)}$ {  as in the proof \eqref{e:V-H}}.
\end{proof}

\subsection{Conformal estimates}
The final ingredients for the proof of Theorem \ref{thm:global} are the following conformal estimates which are qualitatively contained in \cite{astala-prats-saksman}. The first one is a quantification of \cite{astala-prats-saksman}*{Theorem 1.2}.
\begin{lemma}\label{lemma:conf}
Let $O_1$ and $O_2$ be simply connected $\Bes_p$ domains, and $g$ conformally map $O_1$ onto $O_2$. Then $\log g' \in W^{1,p}(O_1)$ with
	\begin{align}
	\label{e:g} 
		\norm{\left[ \log g' \right]' }_{L^{q}(O_1)} 
		\lesssim 1 + \norm{O_1}_{\Bes_q} + \norm{O_2}_{\Bes_q}  \quad 1 < q \le p.
	\end{align}
\end{lemma}
\begin{proof}
It is shown in \cite{astala-prats-saksman}*{Theorem 1.2} that if $h_j$ is a conformal mapping from $\mathbb D$ onto $O_j$, then $h_j', \log h_j' \in W^{1,p}(\mathbb D)$ for a suitable branch of the logarithm. To track the dependence on $\norm{O_j}_{\Bes_p}$, let us outline their argument. For $\e^{\i t} \in \mathbb T$, it is not hard to compute that 
	\[ \arg N_{O_j}(h_j(\e^{\i t})) = -\arg h_j'(\e^{\i t}) + t \]
so that $1+\norm{O_j}_{\Bes_p} \sim \norm{\arg h_j'}_{B^{1-\frac{1}{q}}_{q,q}(\mathbb T)}$ for any $q>1$. Furthermore, the Herglotz extension maps $B^{1-\frac{1}{q}}_{q,q}(\mathbb T) \to W^{1,q}(\mathbb D)$ \cite{triebel83}*{Theorem 4.3.3}, which together with the well-known Herglotz representation \cite{pommerenke-book}*{Theorem 3.3.2}, yields
	\begin{equation}\label{e:logh} \log h_j' = \log \abs{h_j'(0)} + H_j, \quad \norm{H_j}_{W^{1,q}(\mathbb D)} \lesssim 1+\norm{O_j}_{\Bes_q}.\end{equation}
Now $g$ can be factored as $h_2 \circ h_1^{-1}$. By the chain rule and inverse function theorem, 
	\[ \left[ \log g' \right]' = \frac{\tilde H}{h_1'}  \circ h_1^{-1}, \quad \tilde H= \left[\log h_2' \right]' - \left[\log h_1' \right]'.\]
So, changing variables
	\[ \int_O \abs{ \left[ \log g' \right]' }^q = \int_{\mathbb D} \abs{ \tilde H }^q \abs{h_1'}^{2-q}.\]
Because $O_1$ is Dini-smooth, $h_1'$ is bounded above and below \cite{pommerenke-book}*{Theorem 3.3.5}, so the proof is concluded by appealing to \eqref{e:logh}.
\end{proof}

Immediately from Lemma \ref{lemma:conf} and Theorem A, we obtain a quantitative version of \cite{astala-prats-saksman}*{Theorem 1.1} in the Sobolev case. 
\begin{proposition}\label{prop:principal}
Let $F \in W^{1,2}_{\mathrm{loc}}(\Omega)$ be a homeomorphism satisfying
	\[ \overline{\partial} F(z) = \mu(z) \partial F(z), \quad z \in \Omega,\]
such that $F(\Omega)$ is also a bounded $\Bes_p$ domain. Then, 
	\begin{align}
	\label{e:Fp}&\norm{D \log \partial F}_{L^p(\Omega)} \le \left[\norm{F(\Omega)}_{\Bes_p} + \left(1+ \norm{\Omega}_{\Bes_p} \right)\left(1+\norm{\mu}_{W^{1,p}(\Omega)}^3 \right) \right] \mathcal E\left(\norm{\mu}_{W^{1,r}(\Omega)} \right); \\
	\label{e:F2}&\norm{ D \log \partial F}_{L^2(\Omega)} \lesssim 1 + \norm{F(\Omega)}_{\Bes_2} + \norm{\Omega}_{\Bes_2} \mathcal E\left( \norm{\mu}_{W^{1,r}(\Omega)} \right).
	\end{align}
\end{proposition}
\begin{proof}
Let $E\mu$ be an extension of $\mu$ from $\Omega$ to $\mathbb{C}$ satisfying
	\[ \norm{E\mu}_\infty \le \varkappa(k) < 1, \quad \norm{E\mu}_{W^{1,q}(\mathbb C)} \lesssim \norm{\mu}_{W^{1,q}(\Omega)}, \quad 2 \le q \le p.\]
The difficulty in constructing $E\mu$ is to guarantee the first property. To do so, one needs only to modify the parameters in the usual first order extension \cite{evans}*{Theorem 5.4.1}. To demonstrate, extending over the line $x_n=0$, one can use
	\[ Eu(x_1,\ldots,x_n) = -\ep u\left(x_1,\ldots,x_{n-1},-\tfrac{3}{2\ep} x_n\right) + (1+\ep) u\left(x_1,\ldots,x_{n-1},-\tfrac{1}{2(1+\ep)} x_n\right),\]
where $\ep>0$ is chosen small enough that $k(1+2\ep)=\varkappa(k)<1$. Composing with $\C^1$ boundary parameterizations does not affect the $L^\infty$ norm of the extension and will only add a constant to the relevant Sobolev norms.  {  Since $\mu$ has compact support, so does $E\mu$.}

Let $G$ be the principal solution to $\overline{\partial}G = (E\mu) \partial G$. $G$ has the well-known solution formula $G(z) = z + \mathcal K(\rho)$ where $(I-(E\mu) \B)\rho=E\mu$. By Theorem \ref{thm:local}, 
	\begin{equation}\label{eq:F} \|DG\|_{W^{1,s}(\mathbb C)} \lesssim 1+\|E\mu\|_{W^{1,s}(\mathbb C)}^3 \lesssim 1+ \|\mu\|_{W^{1,s}(\Omega)}^3, \quad s \in \{r,p\}.\end{equation}
{ Furthermore, $G$ is bilipschitz by the discussion in Section \ref{ss:bilip} with bounds
	\[ \mathcal E \left( - \norm{\mu}_{W^{1,r}(\Omega)} \right) \lesssim \abs{ DG(z)} \lesssim 1 + \|\mu\|_{W^{1,r}(\Omega)}^3, \quad \text{for all} \ z \in \mathbb C.\]} 
Since $\Omega$ is a $\Bes_p$ domain and $G \in W^{2,p}(\mathbb C)$, standard trace results \cites{gagliardo,triebel83} together with \eqref{eq:F} show that $G(\Omega)$ is also a $\Bes_{p}$ domain and moreover the estimates 
	\begin{align}
		\label{e:trace}
		\begin{aligned}
			&\norm{G(\Omega)}_{\Bes_2} \lesssim 
			\norm{\Omega}_{\Bes_2} \mathcal E\left( \norm{\mu}_{W^{1,r}(\Omega)} \right), \\
			&\norm{G(\Omega)}_{\Bes_p} \lesssim 
			\norm{\Omega}_{\Bes_p} \left( 1+\norm{\mu}_{W^{1,p}(\Omega)}^3 \right)\mathcal E\left( \norm{\mu}_{W^{1,r}(\Omega)} \right)
		\end{aligned}
	\end{align}
hold.
Introducing $g = F \circ G^{-1}$, $g$ is conformal on $G(\Omega)$ with $g(G(\Omega)) =F(\Omega)$. Furthermore, 
	\begin{equation}\label{eq:dlogdf} \partial \log \partial F= \frac{\partial^2 F}{\partial F} = \frac{(g' \circ G) \partial^2 G}{ (g' \circ G) \partial G} + \frac{(g'' \circ G) (\partial G)^2}{ (g' \circ G) \partial G} \eqqcolon F_1 + F_2. \end{equation}
Recall that $\partial G = \e^\sigma$, for $\sigma$ satisfying $\overline{\partial} \sigma = (E\mu) \partial \sigma + \partial(E\mu)$. Noticing that $F_1 = \partial \sigma$ and applying \eqref{eq:sigma-reg}, we have
	\begin{equation}\label{e:F1} \norm{F_1}_{L^q(\mathbb C)} \le \norm{\sigma}_{W^{1,q}(\mathbb C)} \lesssim \|E\mu\|_{W^{1,q}(\mathbb C)} \lesssim \|\mu\|_{W^{1,q}(\Omega)}, \quad 2 \le q \le p,\end{equation}
which obeys the estimates \eqref{e:Fp} and \eqref{e:F2}. To estimate $F_2$, change variables and use the quasiconformality of $G$ in \eqref{e:Jf-quasi} to obtain for any $2 \le q \le p$,
	\begin{equation}\label{e:loggG} \int_\Omega \left|F_2 \right|^q \lesssim \int_{G(\Omega)} \left|\frac{g''}{ g'} \right|^q \left( |JG| \circ G^{-1} \right)^{\frac{q}{2}-1} \le \norm{ \left[\log g'\right]' }_{L^{q}(G(\Omega))}^q \|JG\|_\infty^{\frac{q}{2}-1}. \end{equation}
For $q \in \{2,p\}$, estimate the first factor by \eqref{e:g} from Lemma \ref{lemma:conf} and \eqref{e:trace}. When $q=2$, the $JG$ term vanishes in \eqref{e:loggG}, so \eqref{e:F2} is established.
When $q=p$, estimate $JG$ by the Sobolev embedding and \eqref{eq:F} with $s=r$.
\end{proof}

\subsection{Proof of Lemma \ref{lemma:sharp}}\label{ss:sharp}
By \eqref{e:standard} and \eqref{e:NO}, the assumptions of this lemma imply $f \in W^{2,p}(\Omega)$. {  This implies first that $\abs{Df}$ belongs to $L^\infty(\Omega)$ by Sobolev embedding. If we can show $|Df|^{-1}$ also belongs to $L^\infty(\Omega)$ and that $O=f(\Omega)$ is a $\Bes_p$ domain, then the weighted estimate for $\B_{O}$ will follow from Lemma \ref{lemma:Bp}.i. To establish the fact concerning $\abs{Df}^{-1}$ we appeal to the main result of \cite{mateu-extra}. Since $\Omega$ is a $\Bes_p$ domain, embedding of Besov spaces shows that $\Omega$ in fact has $\C^{1+\epsilon}$ boundary for some small $\epsilon>0$. Therefore, by \cite[Theorem, p. 403]{mateu-extra}, $f$ is bilipschitz. We can now conclude by trace regularity that $O=f(\Omega)$ is also $\Bes_p$ since $f$ is bilipschitz and belongs to $W^{2,p}(\Omega)$.} 
\qed

\subsection{Proof of Theorem \ref{thm:global}}{ We follow the same path as the proof of Theorem \ref{thm:local}. In particular, we focus on proving only the lower bound \eqref{e:lower-bound} and the invertibility follows by the same argument. Also note that here we only consider the case $p>2$. Let $g \in \C_0^\infty(\mathbb C)$ and set $w = \mathcal K( \mathbf 1_{\overline{\Omega}} g)$ so that $w$ satisfies the inhomogeneous Beltrami equation $\overline{\partial} w = \mu \partial w + h$ on $\mathbb C$ with $h=(I-\mu \B_\Omega)(\mathbf 1_{\overline{\Omega}} g)$. Notice that $\overline{\partial}w=\mathbf 1_{\overline{\Omega}}g \in \C^\infty(\overline{\Omega})$ and since $\Omega$ is a $\Bes_p$ domain, $\partial w$ and $h$ both belong to $W^{1,p}(\Omega)$. Let us normalize so that $\norm{h}_{W^{1,p}(\Omega)}=1$. Then the precise estimate that remains to be proved is
	\begin{align}\label{e:goal} &\left\Vert \overline{\partial} w \right \Vert_{W^{1,p}(\Omega)} \lesssim \Cp \norm{\B_O}_{\mathcal L(W^{1,p}(O,\omega))} + \Cp^2\left( 1+ \norm{\Omega}_{\Bes_p} \right) \norm{O}_{\Bes_p} , \\
	\label{e:C}& 
	 \Cp = \norm{O}_{\Bes_p} + \left( 1 + \norm{\Omega}_{\Bes_p}\right) \left( 1 + \norm{\mu}_{W^{1,p}}^3 \right).  \end{align}
We remind the reader that we are suppressing the dependence on lower-order terms in accordance with the remarks in \S \ref{ss:assumptions}.} By Sobolev embedding and interpolation, since $p>2$,
	\begin{equation}\label{e:emb} \|h\|_{L^s(\Omega)} \lesssim_{p,s} 1, \quad p\le s\le \infty.\end{equation}
Defining $u=w \circ f^{-1}$, we again have the identities from \eqref{eq:h}, \eqref{eq:wu}, and \eqref{e:wV},
	\begin{equation}\label{e:chainO} 
		\begin{aligned}
		\overline{\partial} w &= (\partial u \circ f) \overline{\partial} f + H, & H &\coloneqq \frac{h}{1-|\mu|^2}= \left(\overline{\partial} u \circ f\right)\overline{ \partial f} \\
		\partial w &= V + \overline{\mu} H,& V &\coloneqq  (\partial u \circ f) \partial f.
		\end{aligned}
		\end{equation}
Let $q \in \left\{ \tfrac{pp'}{2},p\right\}$. As before, introduce $\sigma=\log \partial f$. One of the main differences between this proof and the proof of Theorem \ref{thm:local} is that our estimate for $D\sigma$ comes from the conformal estimate \eqref{e:Fp} in Proposition \ref{prop:principal}. Therefore, 
	\begin{align}
		\label{e:H-sigma-O} \left\Vert H \right\Vert_{W^{1,q}(\Omega)}  &\lesssim 1+ \|\mu\|_{W^{1,p}(\Omega)} , \quad \norm{D\sigma}_{L^q(\Omega)} \lesssim \Cp. 
	 \end{align}
The first key estimate concerning these quantities is the analogue of \eqref{e:V-H}.
\begin{equation}\label{e:V-H-O} \norm{V}_{L^s(\mathbb C)} \lesssim \mathcal E(s) \norm{H}_{L^s(\mathbb C)}, \quad 2 \le s < \infty.\end{equation}
The proof of \eqref{e:V-H-O} follows exactly the lines of Proposition \ref{prop:LpBO} and thus is omitted. In light of the relation $\overline{\partial} f = \mu \partial f$, \eqref{e:V-H-O} provides an estimate for $\norm{\overline{\partial}w}_{L^q(\Omega)}$. Indeed, from \eqref{e:chainO}, $\abs{\overline{\partial}w } \lesssim \abs{V} + \abs{H}$. Thus, it remains to estimate the homogeneous norm $\norm{D \overline{\partial} w}_{L^q(\Omega)}$ in \eqref{e:goal}. To this end, by the product rule and \eqref{e:chainO}, 
	\[ D \overline{\partial} w = U_1 + U_2 + DH, \quad U_1 = D\left( \partial u \circ f \right) \overline{\partial} f, \quad U_2 = \left( \partial u \circ f \right) D\overline{\partial} f.\]
Therefore, once we can prove
\begin{equation}
	\label{e:U1-G-O}\norm{U_1}_{L^q(\Omega)} \lesssim \Cp \norm{\B_O}_{\mathcal L(W^{1,q}(O,\omega))} , \quad \omega = \abs{Jf^{-1}}^{1-q},\end{equation}
we can follow the same bootstrapping algorithm (first for $q<p$ and then taking $q=p$) as in Theorem \ref{thm:local} to conclude the proof of \eqref{e:goal}. Let us first outline the bootstrapping argument and then focus on establishing \eqref{e:U1-G-O}. By virtue of \eqref{e:U1-G-O} we can summarize the proof thus far by
	\begin{equation} 
	\label{e:summ4} \|\overline{\partial} w\|_{W^{1,q}(\Omega)} \lesssim \Cp \left\Vert \B_O \right\Vert_{\mathcal L(W^{1,q}(O,\omega))} + \|U_2\|_{L^q(\Omega)}, \quad q \in \left\{ \tfrac{pp'}{2},p\right\} . 
	\end{equation} 
Let us first consider $q= \frac{pp'}{2}<p$. By Lemma \ref{lemma:Bp}.iii,
	\begin{equation}\label{e:iii} \left\Vert \B_O \right\Vert_{W^{1,q}(O,\omega)} \lesssim  1 + \norm{O}_{\Bes_p}, \quad q = \frac{pp'}{2}.\end{equation}
Furthermore, recall the identity $U_2=V\lambda$ from \eqref{e:U2Vlambda} with $\lambda=D \mu + \mu D \sigma$. By \eqref{e:H-sigma-O}, $\lambda$ satisfies $\norm{\lambda}_{L^p(\Omega)} \lesssim \Cp$. Using H\"older's inequality and applying \eqref{e:V-H-O} with $s = \frac{pq}{p-q}=\frac{p^2}{p-2}$, we obtain $\norm{U_2}_{L^q(\Omega)} \lesssim \Cp$. Inserting this estimate and \eqref{e:iii} into \eqref{e:summ4} and applying \eqref{e:NO}, we obtain
 \begin{equation} 
	\label{e:summ5} \|D w\|_{W^{1,q}(\Omega)} \lesssim \left(1+ \norm{\Omega}_{\Bes_q} \right) \|\overline{\partial} w\|_{W^{1,q}(\Omega)} \lesssim \Cp \left(1+ \norm{\Omega}_{\Bes_p} \right)\left( 1 + \norm{O}_{\Bes_p}\right), \quad q = \frac{pp'}{2}. 
	\end{equation} 
Now, recalling $V = \partial w - \overline{\mu}H$ from \eqref{e:chainO}, since $q=\frac{pp'}{2} > 2$, the $L^\infty(\Omega)$ norm of $V$ is controlled by \eqref{e:summ5} providing the alternative estimate for $U_2$ in \eqref{e:summ4} when $q=p$:
	\[ \norm{U_2}_{L^p(\Omega)} \lesssim \norm{V}_{L^\infty(\Omega)} \norm{\lambda}_{L^p(\Omega)} \lesssim \Cp^2 \left(1+ \norm{\Omega}_{\Bes_p} \right)\left( 1 + \norm{O}_{\Bes_p}\right).\]
Applying this to \eqref{e:summ4} with $q=p$ exactly proves \eqref{e:goal}.

\begin{proof}[Proof of \eqref{e:U1-G-O}]
Notice that since $h$ is supported in $\overline{\Omega}$, \eqref{e:chainO} implies that $\overline{\partial} u$ is supported in $\overline{O}$. Therefore, $\mathbf 1_{\overline{O}} \partial u = \B_{O} \overline{\partial} u$ whence, following the same steps as \eqref{e:SobBeurling}, changing variables and using the quasiconformality of $f$ in the form \eqref{e:Jf-quasi},
	\[ \int_\Omega \left \vert D(\partial u \circ f) \overline{\partial} f \right\vert^q \lesssim \int_{O} \left \vert D\B_O(\overline{\partial}u) \right \vert^q \left \vert Jf^{-1} \right \vert^{1-q} \lesssim \norm{\B_O(\overline{\partial}u)}_{W^{1,q}(O,\omega)}^q.\]
Thus, to establish \eqref{e:U1-G-O}, it remains to control $\norm{\overline{\partial}u}_{W^{1,q}(O,\omega)}$ by $\Cp$. The lower order term, $\norm{\overline{\partial}u}_{L^q(O,\omega)}$, by changing variables and referring to \eqref{e:chainO}, satisfies
	\[ \norm{\overline{\partial}u}_{L^q(O,\omega)}^q = \int_\Omega \abs{\overline{\partial}u \circ f}^q |Jf^{-1} \circ f|^{-q} \lesssim \int_O \abs{\overline{\partial}u \circ f}^q |\partial f|^{2q} \lesssim \int_\Omega |h \partial f|^q.\] By \eqref{e:emb}, $\norm{h}_{L^\infty(\Omega)} \lesssim 1$. On the other hand,  $\partial f = 1+\B(I - \mu \B_\Omega)^{-1} \mu$ so by Proposition \ref{prop:LpBO}, $\norm{\partial f}_{L^q(\Omega)} \lesssim \norm{\mu}_{L^q(\Omega)}$. Therefore, the lower order term $\norm{\overline{\partial} u}_{L^p(O,\omega)}$ is definitely controlled by $\Cp$, since $q \le p$. To show the same bound for $\norm{D \overline{\partial}u}_{L^p(O,\omega)}$ we will in fact establish
	\begin{equation}\label{e:UGO} \norm{D \overline{\partial}u}_{L^q(O,\omega)} \lesssim \norm{G}_{L^q(\Omega)}, \quad G = D(\partial u \circ d) \overline{\partial f}.\end{equation}
This will suffice to conclude the proof of \eqref{e:U1-G-O} since one can verify, using the same calculation as \eqref{e:HG}, that
	\[ \norm{G}_{L^q(\Omega)} \lesssim \norm{H}_{W^{1,q}(\Omega)} + \norm{D\sigma}_{L^q(\Omega)} \lesssim \Cp.\] 
By changing variables, \eqref{e:UGO} follows from the pointwise estimate \eqref{e:pointG}, which can be verified again in this setting, since it only relies on the quasiconformality of $f$.
\end{proof}

\section{Muckenhoupt weights and Jacobians of Quasiconformal maps}\label{s:Jf}
In this section we prove Lemmata \ref{lemma:moser} and \ref{lemma:moser2} concerning the Muckenhoupt weight properties of $\abs{Jf}$ and $\abs{Jf^{-1}}$. The main tool is the following critical Sobolev embedding Lemma of Moser-Trudinger type.
\begin{lemma}\label{lemma:exp}
There exists an absolute constant $C$ such that for every $\sigma: \mathbb C \to \mathbb C$ with $D \sigma \in L^2(\Omega)$, $a \in \mathbb R$, and $1<p<\infty$,
	\[ \sup_{Q \subset \mathbb C} \avg{ \abs{\e^{a\sigma}} }_Q \avg{ \abs{\e^{-a\sigma}} }_{\frac{1}{p-1},Q} \le C \exp \left(C \frac{|a|^{{  2}}p}{p-1} \norm{D\sigma}_{L^2({ \Omega})}^{{ 2}} \right).\]
\end{lemma}
\begin{proof}
Using the Taylor formula, we can write, for any cube $Q \subset \mathbb C$ and any $z \in Q$,
	\[ |\sigma(z) - \sigma_Q| \lesssim \int_{Q} \frac{|D\sigma(w)|}{|w-z|} \, \d w, \quad \sigma_Q \coloneqq \frac{1}{|Q|} \int_Q \sigma(w) \, \d w.\]
For any $p>2$, by Young's inequality
	\[ \int_Q \left|\int_{Q} \frac{|D\sigma(w)|}{|w-z|} \, \d w\right|^p \, dz \le \|D \sigma\|_{L^2(Q)}^p \left( \int_{[-\ell(Q),\ell(Q)]^2} |z|^{-r} \, \d z\right)^{p/r}, \quad \frac 12 + \frac 1r = \frac 1p +1.\]
Notice that by definition, $r < 2$ so
	\[ \int_{[-\ell(Q),\ell(Q)]^2} |z|^{-r} \, \d z \,{  \lesssim}\, \frac{\ell(Q)^{2-r}}{2-r}.\]
However, some calculations show that $2-r = \frac 4{p+2}$ and $p/r = p/2+1$ so that
	\[ \frac{1}{|Q|} \int_Q |\sigma(z)-\sigma_Q|^p \, \d z \lesssim \|D\sigma\|_{L^2(Q)}^p  \left(\frac{p+2}{4} \right)^{p/2+1}.\]
For each $a>0$, we can now compute
	\begin{equation}\label{eq:exp-a} \begin{aligned} \frac{1}{|Q|} \int_Q \e^{a|\sigma(z)-\sigma_Q|} \, \d z &= \sum_{k=0}^\infty \frac{a^k}{k!} \frac{1}{|Q|} \int_Q |\sigma(z)-\sigma_Q|^k \, \d z  
	\lesssim \sum_{k=0}^\infty (a \|D\sigma\|_{L^2})^{k} \frac{k^{k/2}}{k!} .
	\end{aligned}\end{equation}
A crude estimate for the power series is given by
	\[ \sum_{k=0}^\infty A^k \frac{k^{k/2}}{k!} \lesssim \exp(CA^2),\]
for some absolute constant $C$, by splitting into even and odd integers, using Stirling's approximation, and the crude estimate $(2k)! \ge (k!)^2$.
So, for any $a \in \mathbb R$ and $p>1$, by \eqref{eq:exp-a},
	\[ \begin{aligned} \frac{1}{|Q|} \int_Q \abs{ \e^{a\sigma} } &\left( \frac{1}{|Q|} \int_Q \abs{ \e^{-\frac{a}{p-1}\sigma} } \right)^{p-1} \\
	&=\frac{1}{|Q|} \int_Q \abs{\e^{a(\sigma-\sigma_Q)} } \left( \frac{1}{|Q|} \int_Q \abs{ \e^{-\frac{a}{p-1}(\sigma-\sigma_Q) } } \right)^{p-1} \\
	&\le \frac{1}{|Q|} \int_Q \e^{|a| \cdot |\sigma-\sigma_Q|} \left( \frac{1}{|Q|} \int_Q \e^{\frac{|a| }{p-1}|\sigma-\sigma_Q|} \right)^{p-1} \\
	&\le C\exp\left(C|a|^2 {  \norm{D\sigma}_{L^2(\Omega)}^{2}}\right) \exp\left(\frac{C|a|^2 {  \norm{D\sigma}_{L^2(\Omega)}^{2}}}{p-1} \right). \end{aligned} \]
\end{proof}
\subsection{Proof of Lemma \ref{lemma:moser}}\label{ss:moser}

As in the proof of Theorem \ref{thm:local}, $\partial f=\e^\sigma$ where $\sigma$ satisfies 
	\[ \overline{\partial} \sigma = \mu \partial { \sigma} + \partial \mu, \quad \mbox{on } \mathbb C.\]
Therefore,
	\[ (I-\mu \B_\Omega) \overline{\partial} \sigma = \partial \mu, \quad \mbox{on } \Omega.\]
However, $\norm{\B}_{\mathcal L(L^2(\mathbb C))}=1$ hence
	\[ \norm{\partial \mu}_{L^2(\mathbb C)} = \norm{ (I-\mu \B) \overline{\partial} \sigma }_{L^2(\mathbb C)} \ge (1-k) \norm{ \overline{\partial} \sigma }_{L^2(\mathbb C)}.\]
Therefore, $\norm{D\sigma}_{L^2(\mathbb C)} \lesssim L$. Now, the Jacobian takes the form
	\[ |J{f}|=(1-|\mu|^2)|\e^{2\sigma}| \sim \left\vert \e^{2\sigma} \right\vert,\] 
so Lemma \ref{lemma:exp} establishes that $\abs{Jf}^a$ belongs to $\A_p(\mathbb C)$ with the estimate
\begin{equation}\label{e:JfA} \left[ \abs{Jf}^a \right]_{\A_p(\mathbb C)} \le C \exp\left(\frac{Cp|a|^2L^2}{p-1} \right). \end{equation}
We next demonstrate that \eqref{e:JfA} implies $\abs{Jf}^a$ belongs to every $\RH_s(\mathbb C)$. Indeed, for each $1<s<\infty$, by H\"older's inequality, for each cube $Q \subset \mathbb C$, setting $p=\frac{s+1}{s}$,
	\[ |Q|^p= \left( \int_Q \abs{Jf}^{\frac{a}{p}}\abs{Jf}^{-\frac{a}{p}}\right)^p  \le \left( \int_Q \abs{Jf}^a \right)\left( \int_Q \abs{Jf}^{-as}\right)^{p-1} \]
	\[ \le \left( \int_Q \abs{Jf}^a \right) \left( \left[ \abs{Jf}^{as} \right]_{\A_2(\mathbb C)} |Q|^2 \left( \int_Q \abs{Jf}^{as}\right)^{{  -1}} \right) ^{p-1}  .\]
Therefore, applying \eqref{e:JfA} and rearranging the above display,
	\begin{equation}\label{e:JfRH} \left[ \abs{Jf}^a \right]_{\RH_s(\mathbb C)} \coloneqq \sup_{Q \ \mathrm{cube} } \avg{ \abs{Jf}^a }_{s,Q}\avg{ \abs{Jf}^a }_{Q}^{-1} \le C \exp \left( C |a|^2 s L^2 \right).\end{equation}
Now, to handle $\abs{J{f^{-1}}}^a$ we use the identity, for any $t \in \mathbb R$,
	\begin{equation}\label{eq:change} \int_Q \abs{J{f^{-1}}}^{t} = \int_{f^{-1}(Q)} |J{f^{-1}} \circ f|^{t} |Jf| = \int_{f^{-1}(Q)} |Jf|^{1-t}.\end{equation} 
Since $f$ is quasiconformal, $f$ and $f^{-1}$ are quasisymmetric (see e.g. \cite{astala-book}*{Corollary 3.10.4}). Therefore, there exist cubes $R,P$ such that $R \subset f^{-1}(Q) \subset P$ with $|R| \sim |P|$. So, we claim that for every $t \in \mathbb R$,
	\begin{equation}\label{eq:area-dist2} \frac{1}{|P|^t} \left( \int_{P} |Jf| \right)^{t-1} \int_{P} |Jf|^{1-t} \lesssim 
\left\{\begin{array}{cc} \exp \left( C (1-t)^2 L^2 \right), & 1-t > 1; \\
		1, & 0 \le 1-t \le 1; \\
		\exp \left( C t(t-1) L^2 \right), & 1-t < 0. \end{array} \right. 
\end{equation}
When $1-t \ge 1$, \eqref{eq:area-dist2} is simply the $\RH_{1-t}$ property of $\abs{Jf}$ from \eqref{e:JfRH}. If $0 \le 1-t \le 1$, then \eqref{eq:area-dist2} follows by H\"older's inequality. Finally, if $t > 1$, apply the $\A_p(\Omega)$ condition for $|Jf|$ with $1-t=\frac{-1}{p-1}$ from \eqref{e:JfA}. Thus \eqref{eq:area-dist2} is established. Finally, for any $a \in \mathbb R$ and $1<p<\infty$, applying \eqref{eq:change} followed by \eqref{eq:area-dist2} with $t=a$ and $t=-\frac{a}{p-1}$,
	\[  \left( \frac{1}{|Q|} \int_Q \abs{J{f^{-1}}}^{a} \right) \left( \frac{1}{|Q|} \int_Q \abs{J{f^{-1}}}^{-\frac{a}{p-1}} \right)^{p-1} \lesssim \frac{|P|^a}{|Q|^a} \left( \frac{|P|^{-\frac{a}{p-1}}}{|Q|^{-\frac{a}{p-1}}} \right)^{p-1} = 1,\]
with the appropriate implicit constant according to \eqref{eq:area-dist2}.

\subsection{Proof of Lemma \ref{lemma:moser2}}\label{ss:moser2}
By Proposition \ref{prop:principal} $\log \partial F \in {  \dot W^{1,2}(\Omega)}$. Since $\Omega$ is $\Bes_p$ it is also $\C^1$ so there exists an extension $\rho \in W^{1,2}(\mathbb C)$. By Lemma \ref{lemma:exp} $\abs{\e^{a\rho}} \in \A_q(\mathbb C)$. The same argument used to extend to $JF^{-1}$ in the proof of Lemma \ref{lemma:moser} in \S \ref{ss:moser} shows that $\abs{JF^{-1}}^a$ belongs to the following $\A_p(O)$ class for $O = F(\Omega)$, defined by finiteness of the characteristic
	\begin{equation}\label{e:JFinvO} \left[ \abs{JF^{-1}}^a \right]_{\A_p(O)} = \sup_{Q \subset \mathbb C} \avg{ \mathbf 1_{O} \abs{JF^{-1}}^a }_{Q} \avg{ \mathbf 1_O \abs{JF^{-1}}^{-a} }_{\frac{1}{p-1},Q}.\end{equation}
An unpublished result of Wolff \cite{wolff-Ap} states that for any measurable set $O$, if $\omega^{1+\ep} \in \A_p(O)$ then there exists $\upsilon \in \A_p(\mathbb C)$ such that $\upsilon=\omega$ on $O$. See \cite{gc-book}*{Theorem IV.5.5} for a proof, \cite{diplinio23domains}*{Lemma 3.8} for a quantitative version, and \cites{holden,kurki22} for two related results. Because \eqref{e:JFinvO} holds for arbitrary $a$, we can apply Wolff's result to obtain the promised $\upsilon$.

\subsection{Explicit dependence in special cases}\label{ss:particular}
Let us compute this precise dependence in the relevant special cases to establish \eqref{e:Jf-Ap-1} and \eqref{e:Jf-Ap-2}.

Let $1<p<\infty$ and denote by $p'=\frac{p}{p-1}$ the H\"older conjugate. The computation $(1-\frac p2)\frac{-1}{p-1}= 1 -\frac{p'}{2}$ reveals the symmetry
	\begin{equation}\label{e:symm-p} \left[ \abs{Jf^{-1}}^{1-\frac{p}{2}} \right]_{\A_p(\mathbb C)}^{\max\{1,\frac{1}{p-1}\}} = \left[ \abs{Jf^{-1}}^{1-\frac{p'}{2}} \right]_{\A_{p'}(\mathbb C)}^{\max\{1,\frac{1}{p'-1}\}}.\end{equation}

Therefore, we can assume $p>2$ so that $\max\{1,\frac{1}{p-1} \}=1$, $1-(1- \frac{p}{2}) >1$, and $\frac 12 \le 1-(1- \frac{p'}{2}) \le 1$ in order to compute
	\[ \left[ \abs{Jf^{-1}}^{1-\frac{p}{2}} \right]_{\A_p(\mathbb C)}^{\max\{1,\frac{1}{p-1} \}} \le C \exp\left(C p^2 L^2\right). \]
Therefore, \eqref{e:Jf-Ap-1} follows by the symmetry \eqref{e:symm-p}.
To prove \eqref{e:Jf-Ap-2}, let $1<p<\infty$, so that $1-(1-r) >1$, and hence
	\[ \left[ \abs{Jf^{-1}}^{1-r} \right]_{\A_r(\mathbb C)}^{\max\left\{1,\frac{1}{r-1}\right\}} \le C \exp\left(C r^2\max\left\{1,\tfrac{1}{r-1}\right\} L^2 \right) \le C \exp\left(C_1 \max\left\{r^2,\tfrac{1}{r-1}\right\} L^2 \right).\]

{ 
\section*{Thanks}
The authors thank the referee for an extremely thorough report as well as very helpful comments and suggestions.}
\bibliography{refs-beltrami}
\bibliographystyle{amsplain}

\end{document}